\pdfoutput=1
\documentclass[a4paper,11pt]{amsart}

\usepackage[margin=1in,includehead,includefoot]{geometry}
\usepackage[utf8]{inputenc}
\usepackage[T1]{fontenc}
\usepackage{lmodern,microtype,calc}
\usepackage{mathtools,amsthm,amssymb}
\usepackage{enumitem}
\usepackage{tikz}
\usetikzlibrary{decorations.pathmorphing}
\usepackage{hyperref,bookmark}

\hypersetup{colorlinks=true,linkcolor=black,citecolor=black,filecolor=black,urlcolor=black}

\makeatletter
\let\citationorig\citation
\def\citation#1{\citationorig{#1}\@for\@tempa:=#1\do{\@ifundefined{cit@\@tempa}{\global\@namedef{cit@\@tempa}{}}{}}}
\let\bibitemorig\bibitem
\def\bibitem#1{\@ifundefined{cit@#1}{\typeout{LaTeX Warning: Unused bibitem `#1'}}{}\bibitemorig{#1}}
\let\old@setaddresses\@setaddresses
\def\@setaddresses{\medskip{\parindent 0pt\let\scshape\relax\let\ttfamily\relax\old@setaddresses}}
\makeatother

\renewenvironment{itemize}{\begin{itemorig}[label=\textbullet, noitemsep, topsep=3pt plus 3pt, leftmargin=1.5em]}{\end{itemorig}}

\renewenvironment{enumerate}{\begin{enumorig}[label={\upshape(\arabic*)}, noitemsep, topsep=3pt plus 3pt, leftmargin=*]}{\end{enumorig}}
\newenvironment{enumeratea}{\begin{enumorig}[label={\upshape(A)}, noitemsep, topsep=3pt plus 3pt, leftmargin=*]}{\end{enumorig}}
\newenvironment{enumerateb}{\begin{enumorig}[label={\upshape(B)}, noitemsep, topsep=3pt plus 3pt, leftmargin=*]}{\end{enumorig}}
\newenvironment{enumerates}{\begin{enumorig}[label={\upshape(S\arabic*)}, noitemsep, topsep=3pt plus 3pt, leftmargin=*]}{\end{enumorig}}

\newtheorem{thmmain}{Theorem}
\newtheorem{theorem}{Theorem}[section]
\newtheorem{lemma}[theorem]{Lemma}
\newtheorem{corollary}[theorem]{Corollary}

\let\close\rhd
\let\leq\leqslant
\let\geq\geqslant
\let\setminus\smallsetminus
\let\Sigma\varSigma

\def\upset{\mathord{\uparrow}}
\def\downset{\mathord{\downarrow}}

\def\low#1{[#1]}

\DeclareMathOperator\Inc{Inc}
\DeclareMathOperator\id{id}

\def\setN{\mathbb{N}}

\def\calA{\mathcal{A}}
\def\calB{\mathcal{B}}
\def\calD{\mathcal{D}}
\def\calE{\mathcal{E}}
\def\calG{\mathcal{G}}
\def\calK{\mathcal{K}}
\def\calL{\mathcal{L}}
\def\calT{\mathcal{T}}
\def\calV{\mathcal{V}}

\allowdisplaybreaks

\hypersetup{
  pdftitle={Minors and dimension},
  pdfauthor={Bartosz Walczak}
}

\title{Minors and dimension}

\author{Bartosz Walczak}

\address{Department of Theoretical Computer Science, Faculty of Mathematics and Computer Science, Jagiellonian University, Kraków, Poland}
\email{\href{mailto:walczak@tcs.uj.edu.pl}{walczak@tcs.uj.edu.pl}}

\thanks{A journal version of this paper appeared in \href{https://doi.org/10.1016/j.jctb.2016.09.001}{\emph{J.\ Comb.\ Theory Ser.~B} 122, 668--689, 2017}.}
\thanks{A conference version of this paper appeared in: \href{https://doi.org/10.1137/1.9781611973730.113}{Piotr Indyk (ed.), \emph{26th Annual ACM-SIAM Symposium on Discrete Algorithms (SODA 2015)}, pp.~1698--1707, SIAM, Philadelphia, 2015}.}
\thanks{The author was partially supported by National Science Center of Poland grant 2011/03/N/ST6/03111.}

\begin{document}

\begin{abstract}
It has been known for 30 years that posets with bounded height and with cover graphs of bounded maximum degree have bounded dimension.
Recently, Streib and Trotter proved that dimension is bounded for posets with bounded height and planar cover graphs, and Joret et~al.\ proved that dimension is bounded for posets with bounded height and with cover graphs of bounded tree-width.
In this paper, it is proved that posets of bounded height whose cover graphs exclude a fixed topological minor have bounded dimension.
This generalizes all the aforementioned results and verifies a conjecture of Joret et~al.
The proof relies on the Robertson-Seymour and Grohe-Marx graph structure theorems.
\end{abstract}

\maketitle

\section{Introduction}

In this paper, we are concerned with finite partially ordered sets, which we simply call \emph{posets}.
The \emph{dimension} of a poset $P$ is the minimum number of linear orders that form a \emph{realizer} of $P$, that is, their intersection gives rise to $P$.
The notion of dimension was introduced in 1941 by Dushnik and Miller \cite{DM41} and since then has been one of the most extensively studied parameters in the combinatorics of posets.
Much of this research has been focused on understanding when and why dimension is bounded, and this is also the focus of the current paper.
The monograph \cite{Tro-book} contains a comprehensive introduction to poset dimension theory.

To some extent, dimension for posets behaves like chromatic number for graphs.
There is a natural construction of a poset with dimension $d$, the \emph{standard example} $S_d$ (see Figure \ref{fig:examples}), which plays a similar role to the complete graph $K_d$ in the graph setting.
Every poset that contains $S_d$ as a subposet must have dimension at least $d$.
On the other hand, there are posets of arbitrarily large dimension not containing $S_3$ as a subposet, just as there are triangle-free graphs with arbitrarily large chromatic number.
Moreover, it is NP-complete to decide whether a poset has dimension at most $d$ for any $d\geq 3$ \cite{Yan82}, just as it is for the chromatic number.

These similarities motivated research on how the dimension of a poset depends on its ``graphic structure''.
There are two natural ways of deriving a graph from a poset: the \emph{comparability graph} connects any two comparable elements, while the \emph{cover graph} connects any two elements that are comparable and whose comparability is not implied by other comparabilities and by transitivity of the order.
It is customary to include only the cover graph edges in drawings of posets and to describe the ``topology'' of a poset in terms of its cover graph rather than its comparability graph.
This choice has a clear advantage: posets of large height still can have sparse ``topology''.

The above-mentioned analogy of poset dimension to graph chromatic number suggests that posets with sparse ``topology'' should have small dimension.
In this vein, Trotter and Moore \cite{TM77} showed that posets whose cover graphs are trees have dimension at most $3$.
How about planarity---can we expect a property of posets similar to the famous four-color theorem?
There is no strict analogy: Trotter \cite{Tro78} and Kelly \cite{Kel81} constructed posets with planar cover graphs that still contain arbitrarily large standard examples as subposets and thus have arbitrarily large dimension (see Figure \ref{fig:examples}).
Actually, Kelly's construction gives planar posets (a \emph{planar poset} admits a drawing such that the vertical placement of points agrees with their poset order).
These constructions led researchers to abandon the study of connections between the dimension of a poset and the structure of its cover graph for decades.

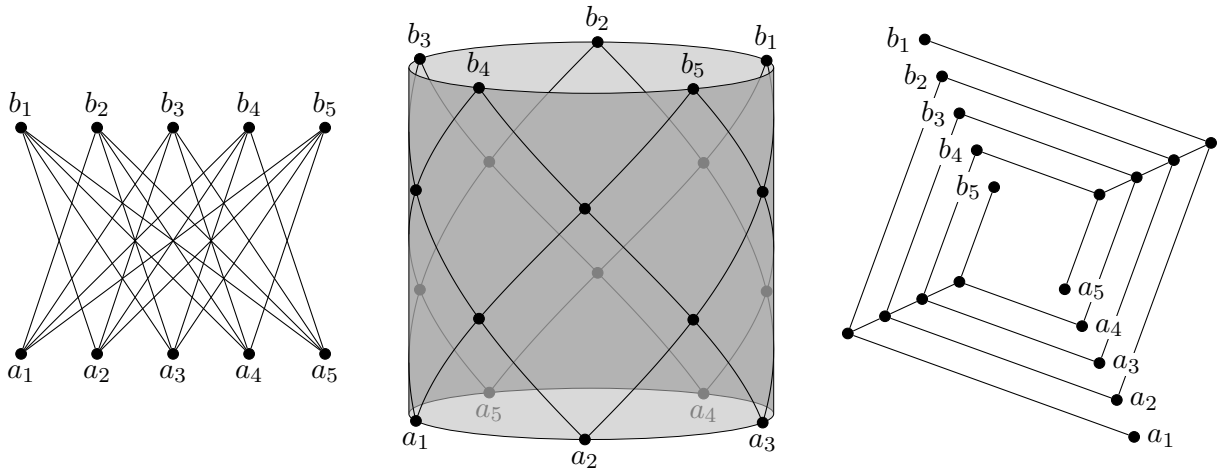
\begin{figure}[t]
\centering
\begin{tikzpicture}[xscale=1,yscale=3,baseline=(current bounding box.center)]
\tikzstyle{every node}=[circle,draw,fill,minimum size=4pt,inner sep=0pt]
\tikzstyle{every label}=[rectangle,draw=none,fill=none,inner sep=0pt,label distance=2pt]
\node[label=below:$a_1$] (a1) at (0,0) {};
\node[label=below:$a_2$] (a2) at (1,0) {};
\node[label=below:$a_3$] (a3) at (2,0) {};
\node[label=below:$a_4$] (a4) at (3,0) {};
\node[label=below:$a_5$] (a5) at (4,0) {};
\node[label=above:$b_1$] (b1) at (0,1) {};
\node[label=above:$b_2$] (b2) at (1,1) {};
\node[label=above:$b_3$] (b3) at (2,1) {};
\node[label=above:$b_4$] (b4) at (3,1) {};
\node[label=above:$b_5$] (b5) at (4,1) {};
\path (a1) edge (b2) edge (b3) edge (b4) edge (b5);
\path (a2) edge (b1) edge (b3) edge (b4) edge (b5);
\path (a3) edge (b1) edge (b2) edge (b4) edge (b5);
\path (a4) edge (b1) edge (b2) edge (b3) edge (b5);
\path (a5) edge (b1) edge (b2) edge (b3) edge (b4);
\end{tikzpicture}\hfill
\begin{tikzpicture}[xscale=2.4,yscale=0.85,baseline=(current bounding box.center)]
\tikzstyle{every node}=[circle,draw,fill,minimum size=4pt,inner sep=0pt]
\draw[smooth,domain=-2:106] plot ({sin(\x)}, {0.05*(\x+2)-0.4*cos(\x)});
\draw[smooth,domain=70:178] plot ({sin(\x)}, {0.05*(\x-70)-0.4*cos(\x)});
\draw[smooth,domain=214:322] plot ({sin(\x)}, {0.05*(\x-214)-0.4*cos(\x)});
\draw[smooth,domain=2:110] plot ({-sin(\x)}, {0.05*(\x-2)-0.4*cos(\x)});
\draw[smooth,domain=74:182] plot ({-sin(\x)}, {0.05*(\x-74)-0.4*cos(\x)});
\draw[smooth,domain=218:326] plot ({-sin(\x)}, {0.05*(\x-218)-0.4*cos(\x)});
\fill[black!15] (0,0) ellipse (1cm and 0.4cm);
\fill[black!15] (0,5.4) ellipse (1cm and 0.4cm);
\fill[black!30] (-1,0) arc (180:0:1cm and 0.4cm)--(1,5.4) arc (0:-180:1cm and 0.4cm)--cycle;
\begin{scope}
\clip (-1,0) arc (180:0:1cm and 0.4cm)--(1,5.4) arc (0:-180:1cm and 0.4cm);
\draw[smooth,black!50,domain=90:106] plot ({sin(\x)}, {0.05*(\x+2)-0.4*cos(\x)});
\draw[smooth,black!50,domain=90:178] plot ({sin(\x)}, {0.05*(\x-70)-0.4*cos(\x)});
\draw[smooth,black!50,domain=142:250] plot ({sin(\x)}, {0.05*(\x-142)-0.4*cos(\x)});
\draw[smooth,black!50,domain=214:270] plot ({sin(\x)}, {0.05*(\x-214)-0.4*cos(\x)});
\draw[smooth,black!50,domain=90:110] plot ({-sin(\x)}, {0.05*(\x-2)-0.4*cos(\x)});
\draw[smooth,black!50,domain=90:182] plot ({-sin(\x)}, {0.05*(\x-74)-0.4*cos(\x)});
\draw[smooth,black!50,domain=146:254] plot ({-sin(\x)}, {0.05*(\x-146)-0.4*cos(\x)});
\draw[smooth,black!50,domain=218:270] plot ({-sin(\x)}, {0.05*(\x-218)-0.4*cos(\x)});
\node[black!50] at ({sin(106)}, {1.8-0.4*cos(106)}) {};
\node[black!50] at ({sin(178)}, {1.8-0.4*cos(178)}) {};
\node[black!50] at ({sin(250)}, {1.8-0.4*cos(250)}) {};
\node[black!50] at ({sin(142)}, {3.6-0.4*cos(142)}) {};
\node[black!50] at ({sin(214)}, {3.6-0.4*cos(214)}) {};
\end{scope}
\begin{scope}
\clip (0,5.4) ellipse (1cm and 0.4cm);
\draw[smooth,domain=90:106] plot ({sin(\x)}, {0.05*(\x+2)-0.4*cos(\x)});
\draw[smooth,domain=90:178] plot ({sin(\x)}, {0.05*(\x-70)-0.4*cos(\x)});
\draw[smooth,domain=142:250] plot ({sin(\x)}, {0.05*(\x-142)-0.4*cos(\x)});
\draw[smooth,domain=214:270] plot ({sin(\x)}, {0.05*(\x-214)-0.4*cos(\x)});
\draw[smooth,domain=90:110] plot ({-sin(\x)}, {0.05*(\x-2)-0.4*cos(\x)});
\draw[smooth,domain=90:182] plot ({-sin(\x)}, {0.05*(\x-74)-0.4*cos(\x)});
\draw[smooth,domain=146:254] plot ({-sin(\x)}, {0.05*(\x-146)-0.4*cos(\x)});
\draw[smooth,domain=218:270] plot ({-sin(\x)}, {0.05*(\x-218)-0.4*cos(\x)});
\end{scope}
\begin{scope}
\clip (-1,0) arc (-180:0:1cm and 0.4cm)--(1,5.4) arc (0:-180:1cm and 0.4cm);
\draw[very thin,black!60] (-1,0) arc (180:0:1cm and 0.4cm);
\draw[smooth,domain=-74:34] plot ({sin(\x)}, {0.05*(\x+74)-0.4*cos(\x)});
\draw[smooth,domain=-2:90] plot ({sin(\x)}, {0.05*(\x+2)-0.4*cos(\x)});
\draw[smooth,domain=70:90] plot ({sin(\x)}, {0.05*(\x-70)-0.4*cos(\x)});
\draw[smooth,domain=270:322] plot ({sin(\x)}, {0.05*(\x-214)-0.4*cos(\x)});
\draw[smooth,domain=-70:38] plot ({-sin(\x)}, {0.05*(\x+70)-0.4*cos(\x)});
\draw[smooth,domain=2:90] plot ({-sin(\x)}, {0.05*(\x-2)-0.4*cos(\x)});
\draw[smooth,domain=74:90] plot ({-sin(\x)}, {0.05*(\x-74)-0.4*cos(\x)});
\draw[smooth,domain=270:326] plot ({-sin(\x)}, {0.05*(\x-218)-0.4*cos(\x)});
\end{scope}
\node at ({sin(-38)}, {1.8-0.4*cos(-38)}) {};
\node at ({sin(34)}, {1.8-0.4*cos(34)}) {};
\node at ({sin(-2)}, {3.6-0.4*cos(-2)}) {};
\node at ({sin(70)}, {3.6-0.4*cos(70)}) {};
\node at ({sin(286)}, {3.6-0.4*cos(286)}) {};
\draw[very thin] (-1,5.4)--(-1,0) arc (-180:0:1cm and 0.4cm)--(1,5.4);
\draw[very thin] (0,5.4) ellipse (1cm and 0.4cm);
\tikzstyle{every label}=[rectangle,draw=none,fill=none,inner sep=0pt,label distance=2pt]
\node[label=below:$a_1$] at ({sin(-74)}, {-0.4*cos(-74)}) {};
\node[label=below:$a_2$] at ({sin(-2)}, {-0.4*cos(-2)}) {};
\node[label=below:$a_3$] at ({sin(70)}, {-0.4*cos(70)}) {};
\node[black!50,label={[black!50]below:$a_4$}] at ({sin(142)}, {-0.4*cos(142)}) {};
\node[black!50,label={[black!50]below:$a_5$}] at ({sin(214)}, {-0.4*cos(214)}) {};
\node[label=above:$b_5$] at ({sin(34)}, {5.4-0.4*cos(34)}) {};
\node[label=above:$b_1$] at ({sin(106)}, {5.4-0.4*cos(106)}) {};
\node[label=above:$b_2$] at ({sin(178)}, {5.4-0.4*cos(178)}) {};
\node[label=above:$b_3$] at ({sin(250)}, {5.4-0.4*cos(250)}) {};
\node[label=above:$b_4$] at ({sin(322)}, {5.4-0.4*cos(322)}) {};
\end{tikzpicture}\hfill
\begin{tikzpicture}[scale=.54,rotate=25,baseline=(current bounding box.center)]
\tikzstyle{every node}=[circle,draw,fill,minimum size=4pt,inner sep=0pt]
\tikzstyle{every label}=[rectangle,draw=none,fill=white,inner sep=1pt,label distance=1.5pt]
\node (z1) at (-5,-0.25) {};
\node (z2) at (-4,-0.25) {};
\node (z3) at (-3,-0.25) {};
\node (z4) at (-2,-0.25) {};
\node (w1) at (5,0.25) {};
\node (w2) at (4,0.25) {};
\node (w3) at (3,0.25) {};
\node (w4) at (2,0.25) {};
\node[label={[yshift=-0.5pt]right:$a_1$}] (a1) at (0.25,-5.5) {};
\node[label=left:$b_1$] (b1) at (-0.25,5.5) {};
\node[label={[yshift=-0.5pt]right:$a_2$}] (a2) at (0.25,-4.5) {};
\node[label=left:$b_2$] (b2) at (-0.25,4.5) {};
\path (z1) edge (a1) edge (b2) edge (z2);
\path (w1) edge (b1) edge (a2) edge (w2);
\node[label={[yshift=-0.5pt]right:$a_3$}] (a3) at (0.25,-3.5) {};
\node[label=left:$b_3$] (b3) at (-0.25,3.5) {};
\path (z2) edge (a2) edge (b3) edge (z3);
\path (w2) edge (b2) edge (a3) edge (w3);
\node[label={[yshift=-0.5pt]right:$a_4$}] (a4) at (0.25,-2.5) {};
\node[label=left:$b_4$] (b4) at (-0.25,2.5) {};
\path (z3) edge (a3) edge (b4) edge (z4);
\path (w3) edge (b3) edge (a4) edge (w4);
\node[label={[yshift=-0.5pt]right:$a_5$}] (a5) at (0.25,-1.5) {};
\node[label=left:$b_5$] (b5) at (-0.25,1.5) {};
\path (z4) edge (a4) edge (b5);
\path (w4) edge (b4) edge (a5);
\end{tikzpicture}
\caption{From left to right: standard example $S_5$ ($a_i<b_j$ if and only if $i\neq j$); Trotter's poset with planar cover graph containing $S_5$ as a subposet; Kelly's planar poset containing $S_5$ as a subposet}
\label{fig:examples}
\end{figure}

The topic has regained interest in recent years.
Having noted that the height of the posets produced by the above-mentioned constructions grows with the dimension, Felsner, Li, and Trotter \cite{FLT10} conjectured that posets of bounded height with planar cover graphs have bounded dimension.\footnote{The conjecture is not stated explicitly in \cite{FLT10}, but it is a natural consequence of the results proved therein.}
This was verified in breakthrough work of Streib and Trotter \cite{ST14}.
Then, Joret et~al.\ \cite{JMM+16} proved that dimension is bounded for posets of bounded height whose cover graphs have bounded tree-width.
It became apparent that the connection between poset dimension theory and structural graph theory is much deeper than it had been thought before.
The condition that the height is bounded cannot be removed from the latter result, as the posets in Kelly's construction have tree-width $3$.
However, Biró, Keller, and Young \cite{BKY16} proved that posets with cover graphs of path-width $2$ (and with no restriction on the height) have bounded dimension.
Very recently, this was further generalized by Joret et~al.\ \cite{JMT+17} to posets with cover graphs of tree-width $2$.
Some other results bounding the dimension of posets whose cover graphs have a specific structure are discussed in the introductory sections of \cite{JMM+16}.

Joret et~al.\ \cite{JMM+16} conjectured that posets of bounded height whose cover graphs exclude a fixed graph as a minor have bounded dimension.
This is verified and further generalized to excluded topological minors in the present paper.

\begin{thmmain}
\label{thm:main}
Posets of bounded height whose cover graphs exclude a fixed graph as a topological minor have bounded dimension.
\end{thmmain}

It follows from Theorem \ref{thm:main} that posets whose comparability graphs exclude a fixed graph as a topological minor have bounded dimension.
Indeed, if the comparability graph of a poset $P$ excludes a graph $H$ as a topological minor, then so does the cover graph of $P$ and the height of $P$ is less than the number of vertices of $H$.
This generalizes an old result that posets with comparability graphs of bounded maximum degree have bounded dimension \cite{FK86,RT-unpub}.

The proof of Theorem \ref{thm:main} relies on structural decomposition theorems for graphs excluding a fixed (topological) minor due to Robertson and Seymour \cite{RS03} and Grohe and Marx \cite{GM15}.
The heart of the proof lies in carrying the bounds on dimension through tree decomposition.
The key concept thereof is a poset analogue of a torso---a construction that extends the subposets induced on the bags of the tree decomposition by gadgets of bounded size whose purpose is to imitate the interaction of these subposets with the rest of the poset.
The function bounding the dimension in terms of the height and the excluded topological minor that results from the proof is enormous, and no effort has been made to compute its precise order of magnitude.

After a preliminary version of this paper had been published, Micek and Wiechert \cite{MW17} came up with an alternative proof of Theorem \ref{thm:main}, which avoids the use of graph structure theorems.
Subsequently, Joret, Micek, and Wiechert \cite{JMW18} proved a generalization of Theorem \ref{thm:main} for posets whose cover graphs belong to any class of graphs with bounded expansion.

To bound the height of a poset $P$ is to exclude a long chain as a subposet of $P$.
It is natural to ask what other posets can be excluded instead of a chain in the assumptions of Theorem \ref{thm:main} so that its conclusion remains valid.
An obvious candidate is the standard example $S_d$.
Gutowski and Krawczyk \cite{GK-personal} suggested another candidate---the poset formed by two incomparable chains of size $k$, denoted by $k+k$.
This is motivated by recent results showing that some problems that are ``hard'' for general posets become ``tractable'' for $(k+k)$-free posets.
For instance, any on-line algorithm trying to build a realizer of a poset of width $w$ can be forced to use arbitrarily many linear extensions even when $w=3$ \cite{KMT84}, although posets of width $w$ have dimension at most $w$ \cite{Hir55}.
On the other hand, for $(k+k)$-free posets of width $w$, Felsner, Krawczyk, and Trotter \cite{FKT13} devised an on-line algorithm that builds a realizer of size bounded in terms of $k$ and $w$.
Whether excluding $S_d$ or $k+k$ instead of bounding the height in Theorem \ref{thm:main} or its predecessors keeps the dimension bounded remains a challenging open problem.

\section{Preliminaries}

\subsection{Graph terminology and notation}

We let $V(G)$ and $E(G)$ denote the sets of vertices and edges of a graph $G$.
The subgraph of $G$ induced on a set $X\subseteq V(G)$ is denoted by $G[X]$.

A graph $H$ is a \emph{minor} of a graph $G$ if $H$ can be obtained from $G$ by deleting vertices, deleting edges, and contracting edges, where to contract an edge $uv$ means to replace $u$ and $v$ by a single vertex that becomes connected to all neighbors of $u$ or $v$.
A graph $H$ is a \emph{topological minor} of $G$ if $H$ can be obtained from $G$ by deleting vertices, deleting edges, and contracting edges with at least one endpoint of degree $2$.
A class of graphs $\calG$ is
\begin{itemize}
\item\emph{minor-closed} if every minor of every graph in $\calG$ belongs to $\calG$,
\item\emph{topologically closed} if every topological minor of every graph in $\calG$ belongs to $\calG$,
\item\emph{monotone} if every subgraph of every graph in $\calG$ belongs to $\calG$,
\item\emph{proper} if $\calG$ does not contain all graphs.
\end{itemize}
Every minor-closed class is topologically closed, and every topologically closed class is monotone.
For every graph $H$, the class of graphs excluding $H$ as a minor or a topological minor is proper minor-closed or proper topologically closed, respectively.

A \emph{tree decomposition} of a graph $G$ is a tree $T$ with a mapping $V(T)\ni t\mapsto B_t\subseteq V(G)$ of the nodes of $T$ into subsets of $V(G)$ that satisfies the following two properties:
\begin{itemize}
\item every edge of $G$ is contained in $G[B_t]$ for at least one node $t$ of $T$,
\item for every vertex $v$ of $G$, the set $\{t\in V(T)\colon v\in B_t\}$ forms a non-empty subtree of $T$.
\end{itemize}
The sets $B_t$ are called the \emph{bags} of the tree decomposition.
Since the mapping $t\mapsto B_t$ can assign the same subset of $V(G)$ to more than one node of $T$, we will assume that every bag of $T$ carries the identity of a particular node of $T$ to which it is assigned.
With this in mind, we will identify the nodes of $T$ with the bags of $T$, and we will simply call $T$ a tree decomposition of $G$.

An \emph{adhesion set} of a bag $X$ of $T$ is a set of the form $X\cap Y$ with $XY\in E(T)$.
The \emph{adhesion} of $T$ is the maximum size of an adhesion set of a bag of $T$.
The \emph{torso} of a bag $X$ of $T$ is the graph obtained from $G[X]$ by adding edges between all pairs of vertices in every adhesion set of $X$.

The \emph{tree-width} of a graph $G$ is the minimum number $k$ such that $G$ has a tree decomposition with every bag of size at most $k+1$.
The \emph{radius} of a graph $G$ is the minimum number $r$ such that $G$ has a vertex whose distance from every vertex of $G$ is at most $r$ if $G$ is connected, or it is $\infty$ if $G$ is disconnected.
The \emph{local tree-width} of a graph $G$ is the function $f\colon\setN\to\setN$ such that $f(r)$ is the maximum tree-width of an induced subgraph of $G$ with radius at most $r$.
Here and further on, $\setN$ denotes the set of positive integers.

For a function $f\colon\setN\to\setN$, let $\calL_f$ denote the class of graphs whose all minors have local tree-width bounded by $f$.
For $d\in\setN$, let $\calD_d$ denote the class of graphs with maximum degree at most $d$.
For a class of graphs $\calG$ and for $t\in\setN$, let $\calA_t(\calG)$ denote the class of graphs $G$ such that there is a set $A\subseteq V(G)$ with $|A|\leq t$ and $G\setminus A\in\calG$.
We call a vertex an \emph{apex} to indicate that it can be connected to arbitrary other vertices of the graph.
Hence $\calA_t(\calG)$ is the class of graphs obtained from graphs in $\calG$ by adding at most $t$ apices.
For a class of graphs $\calG$ and for $s\in\setN$, let $\calT_s(\calG)$ denote the class of graphs $G$ that have a tree decomposition $T$ of adhesion at most $s$ such that every torso of $T$ belongs to $\calG$.
The following facts about classes of graphs are straightforward consequences of these definitions:
\begin{itemize}
\item for every $f\colon\setN\to\setN$, the class $\calL_f$ is minor-closed;
\item for every $d\in\setN$, the class $\calD_d$ is topologically closed;
\item for every $t\in\setN$, if $\calG$ is monotone/topologically closed/minor-closed, then $\calA_t(\calG)$ is monotone/topologically closed/minor-closed, respectively;
\item for every $s\in\setN$, if $\calG$ is monotone/topologically closed/minor-closed, then $\calT_s(\calG)$ is monotone/topologically closed/minor-closed, respectively.
\end{itemize}

\subsection{Graph structure theorems}

The classical results of Kuratowski \cite{Kur30} and Wagner \cite{Wag37} assert that the class of planar graphs is characterized by excluding $K_5$ and $K_{3,3}$ as (topological) minors.
For $g\geq 1$, the class of graphs with genus at most $g$ is minor-closed as well, but its complete list of excluded minors is unknown.
Robertson and Seymour, in a monumental series of papers culminating in \cite{RS04}, proved that the list of minimal excluded minors is finite for every minor-closed class of graphs.
One of the key results of this series is a structural decomposition theorem for proper minor-closed classes of graphs $\calG$ \cite{RS03}: every graph in $\calG$ admits a tree decomposition every torso of which is \emph{almost embeddable} in a surface of bounded genus with the exception of a bounded number of apices (the precise definition of \emph{almost embeddable} is not important for this paper).
This is an \emph{approximate} structural characterization of proper minor-closed classes---the class of graphs satisfying the conclusion of the decomposition theorem is also proper minor-closed, although usually much broader than $\calG$.

Grohe \cite{Gro03} proved that graphs almost embeddable in a surface of bounded genus and all minors of these graphs have bounded local tree-width.
This generalizes earlier results of Baker \cite{Bak94} for planar graphs and of Eppstein \cite{Epp00} for graphs of bounded genus.
This also yields an approximate characterization of proper minor-closed classes of graphs that does not involve topology.

\begin{theorem}[Robertson, Seymour \cite{RS03}, Grohe \cite{Gro03}]
\label{thm:grohe}
For every proper minor-closed class of graphs\/ $\calG$, there are\/ $s,t\in\setN$ and\/ $f\colon\setN\to\setN$ such that\/ $\calG\subseteq\calT_s(\calA_t(\calL_f))$.
\end{theorem}

Topologically closed classes of graphs can be substantially richer than minor-closed ones.
In particular, for $d\geq 3$, the class $\calD_d$ of graphs with maximum degree at most $d$ is topologically closed, but every graph is a minor of some graph in $\calD_d$.
Grohe and Marx \cite{GM15} showed that incorporating graphs with bounded maximum degree to the structural decomposition considered above is enough to obtain an approximate characterization of all topologically closed classes of graphs.
Specifically, for any proper topologically closed class $\calG$, they proved that every graph in $\calG$ admits a tree decomposition whose every torso (a) belongs to a fixed proper minor-closed class, or (b) has bounded maximum degree except for a bounded number of apices.
This altogether yields the following approximate characterization of proper topologically closed classes.

\begin{theorem}[Robertson, Seymour \cite{RS03}, Grohe \cite{Gro03}, Grohe, Marx \cite{GM15}]
\label{thm:grohe-marx}
For every proper topologically closed class of graphs\/ $\calG$, there are\/ $s,t,d\in\setN$ and\/ $f\colon\setN\to\setN$ such that\/ $\calG\subseteq\calT_s(\calA_t(\calL_f\cup\calD_d))$.
\end{theorem}

\subsection{Poset terminology and notation}

We let $\leq_P$ and $<_P$ denote the non-strict and strict order relations of a poset $P$.
The \emph{comparability graph} of a poset $P$ is the graph whose vertices represent the elements of $P$ and whose edges represent the strict comparabilities in $P$.
The \emph{height} of a poset $P$ is the maximum size of a chain in $P$, which is the maximum size of a clique in the comparability graph of $P$.
For $x,y\in P$, we say that $y$ \emph{covers} $x$ if $x<_Py$ and there is no $z\in P$ with $x<_Pz<_Py$.
The \emph{cover graph} of $P$ is the graph whose vertices represent the elements of $P$ and whose edges represent the cover relations of $P$.
We will identify $P$ with the vertex set of the cover graph of $P$ and simply call the elements of $P$ \emph{vertices}.
For $x\in P$, we let $\upset_Px=\{y\in P\colon y\geq_Px\}$ and $\downset_Px=\{y\in P\colon y\leq_Px\}$.
A set $X\subseteq P$ is an \emph{up-set} (\emph{down-set}) of $P$ if $X=\bigcup_{x\in X}\upset_Px$ ($X=\bigcup_{x\in X}\downset_Px$, respectively).
The subposet of $P$ induced on a set $X\subseteq P$ is denoted by $P[X]$. 

The \emph{dimension} of a poset $P$ is the minimum number $d$ of linear extensions $L_1,\ldots,L_d$ of $P$ such that $x\leq_Py$ if and only if $x\leq_{L_i}y$ for $1\leq i\leq d$.
An \emph{incomparable pair} of $P$ is an ordered pair of vertices of $P$ that are incomparable in $P$.
We let $\Inc(P)$ denote the set of all incomparable pairs of $P$.
An \emph{alternating cycle} in $\Inc(P)$ is an indexed set $\{(x_i,y_i)\colon 1\leq i\leq k\}\subseteq\Inc(P)$ such that $x_i\leq_Py_{i+1}$ for $1\leq i\leq k$, where the subscripts are assumed to go cyclically over $\{1,\ldots,k\}$.
The following elementary lemma relates linear extensions of $P$ to alternating cycles in $\Inc(P)$.

\begin{lemma}[Trotter, Moore \cite{TM77}]
\label{lem:alt-cycle}
Let\/ $I\subseteq\Inc(P)$.
There is a linear extension\/ $L$ of\/ $P$ such that\/ $x>_Ly$ for every\/ $(x,y)\in I$ if and only if\/ $I$ contains no alternating cycle.
\end{lemma}

\noindent
We call a coloring of a set $I\subseteq\Inc(P)$ \emph{valid} if it makes no alternating cycle in $I$ monochromatic.
As a corollary to Lemma \ref{lem:alt-cycle}, the dimension of $P$ is the minimum number of colors in a valid coloring of $\Inc(P)$.
This characterization of the dimension will be used further in the paper.

\subsection{Preliminary bounds on the dimension}

In the proof of Theorem \ref{thm:main}, we will use the following result already recalled in the introduction.

\begin{theorem}[Joret et~al.\ \cite{JMM+16}]
\label{thm:tree-width}
For any\/ $h,k\in\setN$, posets of height at most\/ $h$ with cover graphs of tree-width at most\/ $k$ have dimension bounded in terms of\/ $h$ and\/ $k$.
\end{theorem}

\noindent
We show how the above can be reproved using the results of this paper at the end of Section \ref{sec:proof}.

Theorem \ref{thm:tree-width} and the next lemma will allow us to conclude that dimension is also bounded for posets of bounded height whose cover graphs and all their minors have bounded local tree-width.
The next lemma is a generalization of a statement implicit in the work of Streib and Trotter \cite{ST14}.
They applied it to planar graphs in their proof that posets with bounded height and planar cover graphs have bounded dimension.

\begin{lemma}
\label{lem:diameter}
Let\/ $\calG$ be a minor-closed class of graphs, let\/ $h,d\in\setN$, and let\/ $\calG_{2h-2}$ denote the class of graphs in\/ $\calG$ with radius at most\/ $2h-2$.
If posets of height at most\/ $h$ with cover graphs in\/ $\calG_{2h-2}$ have dimension at most\/ $d$, then posets of height at most\/ $h$ with cover graphs in\/ $\calG$ have dimension bounded in terms of\/ $h$ and\/ $d$.
\end{lemma}

\noindent
The proof of Lemma \ref{lem:diameter} is a straightforward generalization of the argument of Streib and Trotter \cite{ST14}.
It is presented in Section \ref{sec:diameter} for the reader's convenience.

\begin{corollary}
\label{cor:ltw}
For any\/ $h\in\setN$ and\/ $f\colon\setN\to\setN$, posets of height at most\/ $h$ with cover graphs in\/ $\calL_f$ have dimension bounded in terms of\/ $h$ and\/ $f$.
\end{corollary}

\begin{proof}
The graphs in $\calL_f$ of radius at most $2h-2$ have bounded tree-width, so by Theorem \ref{thm:tree-width}, posets of height at most $h$ that have these graphs as cover graphs have bounded dimension.
The class $\calL_f$ is minor-closed, so the conclusion follows from Lemma \ref{lem:diameter}.
\end{proof}

Finally, we will use the fact that posets of bounded height with cover graphs of bounded maximum degree have bounded dimension, which is a consequence of the following old result.

\begin{theorem}[Rödl, Trotter \cite{RT-unpub}; Füredi, Kahn \cite{FK86}]
\label{thm:bounded-degree}
For every\/ $d\in\setN$, posets with comparability graphs of maximum degree at most\/ $d$ have dimension bounded in terms of\/ $d$.
\end{theorem}

\begin{corollary}
\label{cor:degree}
For any\/ $h,d\in\setN$, posets of height at most\/ $h$ with cover graphs in\/ $\calD_d$ have dimension bounded in terms of\/ $h$ and\/ $d$.
\end{corollary}

\begin{proof}
Whenever $x$ and $y$ are comparable in a poset $P$, there is a path between $u$ and $v$ of length at most $h-1$ in the cover graph of $P$, where $h$ denotes the height of $P$.
Consequently, if the cover graph of $P$ has maximum degree $d$, then the comparability graph of $P$ has maximum degree at most $d^{h-1}$.
The conclusion now follows from Theorem \ref{thm:bounded-degree}.
\end{proof}

\section{Proof of Theorem \ref{thm:main}}
\label{sec:proof}

\subsection{Overview}

The starting point of the proof are Corollaries \ref{cor:ltw} and \ref{cor:degree}, which assert that posets of bounded height with cover graphs in $\calG=\calL_f\cup\calD_d$ have bounded dimension, for any $f\colon\setN\to\setN$ and $d\in\setN$.
We extend the class of graphs $\calG$ that allows us to bound the dimension of posets with bounded height and with cover graphs in $\calG$ first by adding a bounded number of apices, and then by going through tree decomposition.

Adding apices is dealt with in subsection \ref{subsec:apices}.
Although removing $t$ apices from a poset can decrease its dimension by at most $t$ \cite{Hir55}, it can change the cover graph dramatically---many new cover relations can arise by transitivity through the removed apices, and this is the main issue in dealing with apices.

Going through tree decomposition is the main technical content of this paper.
A naive approach would be to try to prove the following statement:
\begin{enumeratea}
\item\label{stat:A} If $P$ is a poset of height at most $h$, $T$ is a tree decomposition of the cover graph of $P$ with adhesion at most $s$, and $P[X]$ has dimension at most $d$ for every bag $X$ of $T$, then the dimension of $P$ is bounded in terms of $h$, $s$ and $d$.
\end{enumeratea}
However, this statement is false already for $h=s=d=2$.
To see this, let $P$ be the set of all $1$-element and $2$-element subsets of $\{1,\ldots,n\}$ ordered by inclusion.
A valid tree decomposition of the cover graph of $P$ with adhesion $2$ is a star with $\{\{1\},\ldots,\{n\}\}$ as the center bag and with the sets $\{\{i\},\{j\},\{i,j\}\}$ for $1\leq i<j\leq n$ as the leaf bags.
It is easy to see that the subposet of $P$ induced on each of these bags has dimension $2$.
On the other hand, it was proved already by Dushnik and Miller \cite{DM41} that the dimension of $P$ is unbounded as $n\to\infty$.
It is worth noting that the statement \ref{stat:A} is true when $s=1$ even with no bound on the height of $P$ \cite{TWW18}.

The difficulty described above should not be surprising.
In the graph setting, properties of the subgraphs induced on the bags of a tree decomposition such as exclusion of a fixed small graph as a (topological) minor do not generalize to the entire graph either.
This is why graph structure theorems such as Theorem \ref{thm:grohe} and Theorem \ref{thm:grohe-marx} deal with \emph{torsos} instead of induced subgraphs.
The additional edges in the torso are used to imitate paths that connect vertices of the bag and pass through vertices outside the bag.
In our poset setting, we enrich the structure of the subposets induced on the bags by defining their \emph{gadget extensions}, and we prove the statement analogous to \ref{stat:A} but assuming that the gadget extensions have bounded dimension:
\begin{enumerateb}
\item\label{stat:B} If $P$ is a poset of height at most $h$, $T$ is a tree decomposition of the cover graph of $P$ with adhesion at most $s$, and the gadget extensions of $P[X]$ have dimension at most $d$ for every bag $X$ of $T$, then the dimension of $P$ is bounded in terms of $h$, $s$ and $d$.
\end{enumerateb}
Gadget extensions are defined in subsection \ref{subsec:gadget}.
They are obtained from the induced subposet $P[X]$ by attaching a \emph{gadget} of bounded size to every adhesion set of $X$.
The role of these gadgets is to imitate the parts of alternating cycles in $\Inc(P)$ that lie outside $X$.

In subsection \ref{subsec:gadget-cover}, we will prove that the gadget extensions of the subposets of $P$ induced on the bags indeed have bounded dimension when the cover graph of $P$ belongs to $\calT_s(\calA_t(\calL_f\cup\calD_d))$.

The statement \ref{stat:B} is proved in subsection \ref{subsec:tree-decomp}, where it is formulated as Lemma \ref{lem:main}.
The goal of the proof is to define a \emph{signature} $\Sigma(x,y)$ for every incomparable pair $(x,y)$ of $P$ so that the number of distinct signatures is bounded and the incomparable pairs with a common signature contain no alternating cycle.
The signatures are the colors in a requested valid coloring of $\Inc(P)$.

Several ideas for constructing the signatures are borrowed from the proof of Theorem \ref{thm:tree-width} in \cite{JMM+16}.
However, that proof heavily relies on the assumption that the bags of the tree decomposition have bounded size and therefore the tree decomposition very accurately describes the entire structure of the cover graph.
Here, the structure of each bag can be very complex, and all we can make use of are valid colorings of the incomparable pairs of the gadget extensions.

Each signature $\Sigma(x,y)$ will include the colors of the incomparable pairs imitating $(x,y)$ in valid colorings of the incomparable pairs of the gadget extensions of the subposets of $P$ induced on a bounded number of carefully selected bags.
The choice of these bags will ensure that every alternating cycle of incomparable pairs of $P$ with a common signature maps into a monochromatic alternating cycle in one of the gadget extensions.
Since the latter cannot exist (the colorings in the gadget extensions are assumed to be valid), we will conclude that the signatures represent a valid coloring of $\Inc(P)$.

\subsection{Apices}
\label{subsec:apices}

This short subsection is devoted to the proof of the following.

\begin{lemma}
\label{lem:apex}
Let\/ $\calG$ be a monotone class of graphs and\/ $h,t,d\in\setN$.
If posets of height at most\/ $h$ with cover graphs in\/ $\calG$ have dimension at most\/ $d$, then posets of height at most\/ $h$ with cover graphs in\/ $\calA_t(\calG)$ have dimension bounded in terms of\/ $h$, $t$ and\/ $d$.
\end{lemma}

\begin{proof}
It is clear that $\calA_t(\calG)=\smash[t]{\calA_1^{(t)}}(\calG)$ (where $\smash[t]{\calA_1^{(t)}}$ denotes the $t$-fold composition of $\calA_1$), so it is enough to consider the case of $t=1$.
Let $P$ be a poset of height at most $h$ with cover graph $G\in\calA_1(\calG)$.
Hence there is a vertex $a\in P$ such that $G\setminus\{a\}\in\calG$.
The cover graphs of $P\setminus\upset_Pa$ and $P\setminus\downset_Pa$ are $G\setminus\upset_Pa$ and $G\setminus\downset_Pa$, respectively.
Both $P\setminus\upset_Pa$ and $P\setminus\downset_Pa$ are subposets of $P$, so they have height at most $h$.
Both $G\setminus\upset_Pa$ and $G\setminus\downset_Pa$ are subgraphs of $G\setminus\{a\}$, so they belong to $\calG$.
It follows that $P\setminus\upset_Pa$ and $P\setminus\downset_Pa$ have dimension at most $d$.
Therefore, there is a valid coloring of $\Inc(P\setminus\upset_Pa)\cup\Inc(P\setminus\downset_Pa)$ with at most $2d$ colors.
Since every element of $\downset_Pa$ is comparable to every element of $\upset_Pa$, the above coloring covers all incomparable pairs of $P\setminus\{a\}$.
Two more colors are enough for the incomparable pairs of $P$ involving $a$---one for those having $a$ as the first member, and the other for those having $a$ as the second member.
\end{proof}

\subsection{Gadget extensions}
\label{subsec:gadget}

Let $P$ be a poset, $T$ be a tree decomposition of the cover graph of $P$, and $Z$ be a bag of $T$.
Let $\calK_Z$ denote the family of adhesion sets of $Z$, that is, $\calK_Z=\{Z\cap Z'\colon ZZ'\in E(T)\}$.
It follows that $\calK_Z$ is a family of cliques in the torso of $Z$ in $T$.
For every vertex $x\in P\setminus Z$, let $K_Z(x)=Z\cap Z'\in\calK_Z$, where $Z'$ is the unique neighbor of $Z$ in $T$ that lies on the $T$-paths from $Z$ to all bags of $T$ containing $x$.

Let $K\in\calK_Z$.
We call an up-set $U$ of $P[Z]$ a \emph{$K$-up-set} if every minimal element of $P[U]$ belongs to $K$ or, equivalently, if $U=\bigcup_{z\in K\cap U}\upset_{P[Z]}z$.
Similarly, we call a down-set $D$ of $P[Z]$ a \emph{$K$-down-set} if every maximal element of $P[D]$ belongs to $K$ or, equivalently, if $D=\bigcup_{z\in K\cap D}\downset_{P[Z]}z$.
The following properties are easy consequences of these definitions and of the definition of $\calK_Z$:
\begin{itemize}
\item if $x\in P\setminus Z$, then $Z\cap\upset_Px$ is a $K_Z(x)$-up-set of $P[Z]$,
\item if $y\in P\setminus Z$, then $Z\cap\downset_Py$ is a $K_Z(y)$-down-set of $P[Z]$.
\end{itemize}
Indeed, if $x\in P\setminus Z$, $z\in Z$, and $x<_Pz$, then the path in the cover graph of $P$ witnessing the comparability of $x$ and $z$ contains a vertex $z'\in K_Z(x)$ such that $x<_Pz'\leq_Pz$, and similarly for the second property.
A $K$-up-set or a $K$-down-set may be empty.

We define two posets, the \emph{weak gadget extension} $P_Z$ of $P[Z]$ and the \emph{strong gadget extension} $P'_Z$ of $P[Z]$, as follows.
To define the common vertex set of $P_Z$ or $P'_Z$, we take $Z$ and add, for every $K\in\calK_Z$, a \emph{gadget}---a set of at most $2^{|K|+1}$ new vertices that consists of
\begin{itemize}
\item a vertex $x_{Z,K,S}$ for every $K$-up-set $S$ of $P[Z]$ such that there is $x\in P\setminus Z$ with $S=Z\cap\upset_Px$ and $K=K_Z(x)$,
\item a vertex $y_{Z,K,S}$ for every $K$-down-set $S$ of $P[Z]$ such that there is $y\in P\setminus Z$ with $S=Z\cap\downset_Py$ and $K=K_Z(y)$.
\end{itemize}
The order $\leq_{P_Z}$ is defined so that
\begin{itemize}
\item $P_Z[Z]=P[Z]$,
\item $x_{Z,K,S}<_{P_Z}z$ for $z\in S$, so that $S=Z\cap\upset_{P_Z}x_{Z,K,S}$,
\item $y_{Z,K,S}>_{P_Z}z$ for $z\in S$, so that $S=Z\cap\downset_{P_Z}y_{Z,K,S}$,
\item $x_{Z,K,S}<_{P_Z}y_{Z,K',S'}$ whenever $S\cap S'\neq\emptyset$.
\end{itemize}
The order $\leq_{\smash[b]{P'_Z}}$ is defined like $\leq_{P_Z}$ except that the last condition above is replaced by the following:
\begin{itemize}
\item $x_{Z,K,S}<_{\smash[b]{P'_Z}}y_{Z,K',S'}$ whenever $K=K'$ or $S\cap S'\neq\emptyset$.
\end{itemize}
It is straightforward to verify that $\leq_{P_Z}$ and $\leq_{\smash[b]{P'_Z}}$ are indeed partial orders.
In particular, every $x_{Z,K,S}$ is a minimal vertex and every $y_{Z,K,S}$ is a maximal vertex of both $P_Z$ and $P'_Z$.
It follows from the definition of $x_{Z,K,S}$ and $y_{Z,K,S}$ that the height of $P_Z$ and $P'_Z$ is at most the height of $P$.

We define mappings $\mu_Z\colon P\to P_Z\,(P'_Z)$ and $\nu_Z\colon P\to P_Z\,(P'_Z)$ as follows:
\begin{align*}
\mu_Z(x)&=\begin{cases}
x_{Z,K,S}&\text{if $x\in P\setminus Z$, where $K=K_Z(x)$ and $S=Z\cap\upset_Px$},\\
x&\text{if $x\in Z$},
\end{cases}\\
\nu_Z(y)&=\begin{cases}
y_{Z,K,S}&\text{if $y\in P\setminus Z$, where $K=K_Z(y)$ and $S=Z\cap\downset_Py$},\\
y&\text{if $y\in Z$}.
\end{cases}
\end{align*}

\begin{lemma}
\label{lem:aux}
The mappings\/ $\mu_Z$ and\/ $\nu_Z$ have the following properties:
\begin{enumerate}
\item\label{item:aux1} If\/ $(x,y)\in\Inc(P)$, then\/ $(\mu_Z(x),\nu_Z(y))\in\Inc(P_Z)$.
\item\label{item:aux2} If\/ $(x,y)\in\Inc(P)$, then\/ $(\mu_Z(x),\nu_Z(y))\in\Inc(P'_Z)$ unless\/ $x,y\notin Z$ and\/ $K_Z(x)=K_Z(y)$.
\item\label{item:aux3} If\/ $x\leq_Py$ and there is\/ $z\in Z$ such that\/ $x\leq_Pz\leq_Py$, then\/ $\mu_Z(x)\leq_{P_Z}\nu_Z(y)$.
\item\label{item:aux4} If\/ $x\leq_Py$, then\/ $\mu_Z(x)\leq_{\smash[b]{P'_Z}}\nu_Z(y)$.
\end{enumerate}
\end{lemma}

\begin{proof}
For the proof of \ref{item:aux1}, suppose $\mu_Z(x)\leq_{P_Z}\nu_Z(y)$.
By the definition of $\mu_Z$ and $\nu_Z$, there is $z\in Z$ such that $x\leq_Pz\leq_Py$, a contradiction.
We cannot have $\mu_Z(x)>_{P_Z}\nu_Z(y)$ either, so $(\mu_Z(x),\nu_Z(y))\in\Inc(P_Z)$.
By \ref{item:aux1} and the definition of $P'_Z$, we directly have \ref{item:aux2}.
For the proof of \ref{item:aux3}, observe that the existence of $z$ and the definition of $\mu_Z$ and $\nu_Z$ yield $\mu_Z(x)\leq_{P_Z}z\leq_{P_Z}\nu_Z(y)$.
It remains to prove \ref{item:aux4}.
If $x,y\notin Z$ and $K_Z(x)=K_Z(y)$, then $\mu_Z(x)<_{\smash[b]{P'_Z}}\nu_Z(y)$ follows from the definition of $P'_Z$.
Now, suppose $x\in Z$, $y\in Z$, or $x,y\notin Z$ and $K_Z(x)\neq K_Z(y)$.
The path in the cover graph of $P$ witnessing the comparability of $x$ and $y$ must pass through a vertex $z\in Z$ such that $x\leq_Pz\leq_Py$.
This and \ref{item:aux3} imply $\mu_Z(x)\leq_{\smash[b]{P'_Z}}\nu_Z(y)$.
\end{proof}

\subsection{Cover graphs of gadget extensions}
\label{subsec:gadget-cover}

Let $P$, $T$ and $Z$ be as in the previous subsection.
Every edge of the cover graph of $P_Z$ which does not belong to the cover graph of $P[Z]$ connects either a vertex $x_{Z,K,S}$ with a minimal element of $P[S]$ or a vertex $y_{Z,K,S}$ with a maximal element of $P[S]$.
The cover graph of $P'_Z$ has the same edges as the cover graph of $P_Z$ and additionally edges connecting pairs of vertices $x_{Z,K,S}$ and $y_{Z,K,S'}$ within the same gadget.
No two vertices from different gadgets are connected by an edge of the cover graph of $P_Z$ or $P'_Z$.
This motivates the following definition.

For a monotone class of graphs $\calG$ and for $s\in\setN$, let $\calE_s(\calG)$ denote the class of all graphs $G$ for which there are a set $Z\subseteq V(G)$, a family $\calK$ of cliques in $G[Z]$, and a partition $\{X_K\}_{K\in\calK}$ of $V(G)\setminus Z$ with the following properties:
\begin{itemize}
\item $G[Z]\in\calG$,
\item $|K|\leq s$ and $|X_K|\leq 2^{s+1}$ for every $K\in\calK$,
\item $E(G)=E(G[Z])\cup\bigcup_{K\in\calK}E(G[K\cup X_K])$, that is, every edge with one of its endpoints in $X_K$ has its other endpoint in $K\cup X_K$.
\end{itemize}
It follows that if $P$ is a poset, $T$ is a tree decomposition of the cover graph of $P$ with adhesion at most $s$, $Z$ is a bag of $T$, and the torso of $Z$ in $T$ belongs to $\calG$, then the cover graphs of the weak and strong gadget extensions $P_Z$ and $P'_Z$ of $P[Z]$ belong to $\calE_s(\calG)$.

\begin{lemma}
\label{lem:gadget-ltw}
For\/ $s,t\in\setN$ and\/ $f\colon\setN\to\setN$, there is\/ $f'\colon\setN\to\setN$ such that\/ $\calE_s(\calA_t(\calL_f))\subseteq\calA_t(\calL_{f'})$.
\end{lemma}

\begin{proof}
Let $G\in\calE_s(\calA_t(\calL_f))$, and let $Z$, $\calK$ and $\{X_K\}_{K\in\calK}$ be as in the definition of $\calE_s(\calA_t(\calL_f))$.
In particular, we have $G[Z]\in\calA_t(\calL_f)$, so there is a set $A\subseteq Z$ with $|A|\leq t$ and $G[Z\setminus A]\in\calL_f$.
Let $H$ be an induced subgraph of $G\setminus A$ with radius at most $r$, where $r$ is a positive integer.
For every $K\in\calK$, the neighborhood in $H$ of the set $X_K\cap V(H)$ is a subset of $K\cap V(H)$, which is a clique in $H$, and hence the subgraph $H[Z\setminus A]$ of $G[Z\setminus A]$ also has radius at most $r$.
Therefore, the graph $H[Z\setminus A]$ has tree-width at most $f(r)$.
Let $T$ be a tree-decomposition of $H[Z\setminus A]$ of width at most $f(r)$.
For every $K\in\calK$, since $K\cap V(H)$ is a clique in $H$, there is a bag $Y_K$ of $T$ such that $K\cap V(H)\subseteq Y_K$.
For every $K\in\calK$, add the set $(K\cup X_K)\cap V(H)$ to $T$ as a new bag that becomes connected to $Y_K$ by a new edge of $T$.
Thus a tree decomposition of $H$ of width at most $\max\{f(r),2^{s+1}+s-1\}$ is obtained.
This shows that $G\setminus A\in\calL_{f'}$ and hence $G\in\calA_t(\calL_{f'})$, where $f'(r)=\max\{f(r),2^{s+1}+s-1\}$ for every $r\in\setN$.
\end{proof}

\begin{lemma}
\label{lem:gadget-degree}
For\/ $s,t,d\in\setN$, there is\/ $d'\in\setN$ such that\/ $\calE_s(\calA_t(\calD_d))\subseteq\calA_t(\calD_{d'})$.
\end{lemma}

\begin{proof}
Let $G\in\calE_s(\calA_t(\calD_d))$, and let $Z$, $\calK$ and $\{X_K\}_{K\in\calK}$ be as in the definition of $\calE_s(\calA_t(\calD_d))$.
In particular, we have $G[Z]\in\calA_t(\calD_d)$, so there is a set $A\subseteq Z$ with $|A|\leq t$ and $G[Z\setminus A]\in\calD_d$.
Let $v$ be a vertex of $G\setminus A$.
Suppose $v\in X_K\setminus A$ for some $K\in\calK$.
The neighborhood of $v$ in $G\setminus A$ is contained in $(K\cup X_K)\setminus A\setminus\{v\}$.
It follows that the degree of $v$ in $G\setminus A$ is at most $2^{s+1}+d$, as $|K\setminus A|\leq d+1$.
Now, suppose $v\in Z$.
Every clique $K\in\calK$ with $v\in K$ is a subset of the closed neighborhood of $v$ in $G[Z]$, which has size at most $d+t+1$.
Therefore, there are at most $2^{d+t}$ cliques $K\in\calK$ with $v\in K$.
Each such clique $K$ gives at most $2^{s+1}$ neighbors of $v$ in $X_K\setminus A$.
Therefore, the total degree of $v$ in $G\setminus A$ is at most $2^{d+t+s+1}+d$.
This shows that $G\setminus A\in\calD_{d'}$ and hence $G\in\calA_t(\calD_{d'})$, where $d'=2^{d+t+s+1}+d$.
\end{proof}

\begin{corollary}
\label{cor:gadget}
For\/ $h,s,t,d\in\setN$ and\/ $f\colon\setN\to\setN$, posets of height at most\/ $h$ with cover graphs in\/ $\calE_s(\calA_t(\calL_f\cup\calD_d))$ have dimension bounded in terms of\/ $h$, $s$, $t$, $f$ and\/ $d$.
In particular, for every poset\/ $P$ of height at most\/ $h$ with cover graph\/ $G\in\calT_s(\calA_t(\calL_f\cup\calD_d))$, the weak and strong gadget extensions\/ $P_Z$ and\/ $P'_Z$ of\/ $P[Z]$ for each bag\/ $Z$ of the tree decomposition of\/ $G$ witnessing\/ $G\in\calT_s(\calA_t(\calL_f\cup\calD_d))$ have dimension bounded in terms of\/ $h$, $s$, $t$, $f$ and\/ $d$.
\end{corollary}

\begin{proof}
By Lemma \ref{lem:gadget-ltw} and Lemma \ref{lem:gadget-degree}, there are $f'\colon\setN\to\setN$ and $d'\in\setN$ such that $\calE_s(\calA_t(\calL_f\cup\calD_d))\subseteq\calA_t(\calL_{f'}\cup\calD_{d'})$.
By Corollary \ref{cor:ltw}, Corollary \ref{cor:degree}, and Lemma \ref{lem:apex}, posets of height at most $h$ with cover graphs in $\calA_t(\calL_{f'}\cup\calD_{d'})$ have dimension bounded in terms of $t$, $f'$ and $d'$.
The second statement follows from the fact that the gadget extensions $P_Z$ and $P'_Z$ have height at most $h$ and their cover graphs belong to $\calE_s(\calA_t(\calL_f\cup\calD_d))$.
\end{proof}

\subsection{Tree decomposition}
\label{subsec:tree-decomp}

This entire subsection is devoted to the proof of the following lemma, which is the heart of the proof of Theorem \ref{thm:main}.

\begin{lemma}
\label{lem:main}
Let\/ $h,s,d\in\setN$.
Let\/ $P$ be a poset of height at most\/ $h$.
Let\/ $T$ be a tree decomposition of the cover graph of\/ $P$ with adhesion at most\/ $s$.
If the weak and strong gadget extensions\/ $P_Z$ and\/ $P'_Z$ of\/ $P[Z]$ for each bag\/ $Z$ of\/ $T$ have dimension at most\/ $d$, then the dimension of\/ $P$ is bounded in terms of\/ $h$, $s$ and\/ $d$.
\end{lemma}

For this entire subsection, let $P$ be a poset of height at most $h$, let $T$ be a tree decomposition of the cover graph of $P$ with adhesion at most $s$, and assume that the weak and strong gadget extensions $P_Z$ and $P'_Z$ of $P[Z]$ over all bags $Z$ of $T$ have dimension at most $d$.

Fix an arbitrary bag as the root of $T$.
We will envision $T$ as a \emph{planted tree}---a tree drawn in the plane so that the root lies at the bottom, the tree grows upwards, and the children of every bag are assigned an order from left to right.
Let $\leq$ denote the bottom-to-top order of $T$.
That is, we have $X\leq Y$ whenever a bag $X$ lies on the path in $T$ from a bag $Y$ down to the root.
Note that $\leq$ is an order on the bags and should not be confused with the order $\leq_P$ of the poset $P$.
For every $x\in P$, let $\low{x}$ denote the unique lowest bag of $T$ containing $x$.

For bags $X$ and $Y$, let $X\close_kY$ denote that there is a sequence of bags $Z_0,\ldots,Z_\ell$ such that $0\leq\ell\leq k$, $Z_0=X$, $Z_\ell=Y$, and there is a sequence of vertices $z_1,\ldots,z_\ell\in P$ such that $z_i\in Z_{i-1}\cap Z_i$ and $\low{z_i}=Z_i$ for $1\leq i\leq\ell$.

\begin{lemma}[cf.\ Lemma 5.2 (1)--(6) in \cite{JMM+16}]
\label{lem:close}
The relation\/ $\close_k$ has the following properties:
\begin{enumerate}
\item\label{item:close1} if\/ $X\close_kY$ and\/ $k\leq\ell$, then\/ $X\close_\ell Y$,
\item\label{item:close2} $X\close_0Y$ if and only if\/ $X=Y$,
\item\label{item:close3} $X\close_{k+\ell}Z$ if and only if there exists a bag\/ $Y$ such that\/ $X\close_kY$ and\/ $Y\close_\ell Z$,
\item\label{item:close4} if\/ $X\close_kY$, then\/ $X\geq Y$,
\item\label{item:close5} if\/ $X\close_kZ$ and\/ $X\geq Y\geq Z$, then\/ $Y\close_kZ$,
\item\label{item:close6} if\/ $X\close_kY$ and\/ $X\close_kZ$, then\/ $Y\close_kZ$ or\/ $Z\close_kY$.
\end{enumerate}
\end{lemma}

\begin{proof}
Properties \ref{item:close1}--\ref{item:close4} follow directly from the definition of $\close_k$.
For the proof of \ref{item:close5}, assume $X>Z$, as otherwise the conclusion is trivial.
Let $Z_0,\ldots,Z_\ell$ be a sequence of bags and $z_1,\ldots,z_\ell$ be a sequence of vertices such that $1\leq\ell\leq k$, $Z_0=X$, $Z_k=Z$, and $z_i\in Z_{i-1}\cap Z_i$ and $\low{z_i}=Z_i$ for $1\leq i\leq\ell$.
Since $X\geq Y\geq Z$, there is an index $i$ such that $Z_{i-1}\geq Y\geq Z_i$.
It follows that $z_i\in Y$, so the sequences $Y,Z_i,\ldots,Z_\ell$ and $z_i,\ldots,z_\ell$ witness $Y\close_kZ$.
To see \ref{item:close6}, observe that $X\close_kY$ and $X\close_kZ$ imply $X\geq Y\geq Z$ or $X\geq Z\geq Y$, and the conclusion follows from \ref{item:close5}.
\end{proof}

\noindent
We will mostly use the properties listed in Lemma \ref{lem:close} implicitly, without a reference.

For every bag $X$, let $\calB_k(X)$ denote the family of bags $Y$ such that $X\close_kY$.
Thus $X\in\calB_k(X)$ and $\calB_k(X)$ lies entirely on the path in $T$ from $X$ down to the root.
It is important to note that $\calB_k(X)$ does not necessarily consist of consecutive bags on this path.

\begin{lemma}[cf.\ Lemma 5.2 (7) in \cite{JMM+16}]
\label{lem:B-size}
For every bag\/ $X$, we have\/ $|\calB_k(X)|\leq 1+s+\cdots+s^k$.
\end{lemma}

\begin{proof}
If $X$ is the root of $T$, then $\calB_k(X)=\{X\}$ and therefore $|\calB_k(X)|=1$ for every $k$.
Now, suppose $X$ is not the root of $T$.
We have $\calB_0(X)=\{X\}$ and therefore $|\calB_0(X)|=1$.
We have $\calB_1(X)\setminus\{X\}=\{\low{z}\colon z\in X$ and $\low{z}\neq X\}$.
If $z\in X$ and $\low{z}\neq X$, then $z\in X\cap Y$, where $Y$ is the bag directly following $X$ on the path in $T$ down to the root.
Since $T$ has adhesion at most $s$, we have $|X\cap Y|\leq s$.
Therefore, we have $|\calB_1(X)\setminus\{X\}|\leq s$ for every bag $X$ of $T$.
Straightforward induction yields $|\calB_k(X)\setminus\calB_{k-1}(X)|\leq s^k$ for $k\geq 1$, whence the lemma follows.
\end{proof}

For any bags $X$ and $Y$ of $T$ such that $\calB_h(X)\cap\calB_h(Y)\neq\emptyset$, let $A(X,Y)$ denote the highest bag in $\calB_h(X)\cap\calB_h(Y)$.

\begin{lemma}[cf.\ Lemma 5.3 in \cite{JMM+16}]
\label{lem:comp}
If\/ $x\leq_Py$, then\/ $\calB_{h-1}(\low{x})\cap\calB_{h-1}(\low{y})\neq\emptyset$, so\/ $A(\low{x},\low{y})$ is defined, and there is\/ $z\in A(\low{x},\low{y})$ such that\/ $x\leq_Pz\leq_Py$.
\end{lemma}

\begin{proof}
Suppose $x\leq_Py$.
Let $z_0,\ldots,z_k$ be a shortest sequence of vertices of $P$ such that
\begin{itemize}
\item $x=z_0<_P\cdots<_Pz_k=y$,
\item $z_i\in\low{z_{i-1}}$ or $z_{i-1}\in\low{z_i}$ for $1\leq i\leq k$.
\end{itemize}
Such a sequence exists---a path in the cover graph of $P$ witnessing the comparability of $x$ and $y$ has these properties.
Since the height of $P$ is at most $h$, we have $k\leq h-1$.
Suppose there are indices $i$ and $j$ with $0<i\leq j<k$ such that $\low{z_{i-1}}<\low{z_i}=\cdots=\low{z_j}>\low{z_{j+1}}$.
Let $Z=\low{z_i}=\cdots=\low{z_j}$.
The second condition on the sequence $z_0,\ldots,z_k$ yields $z_{i-1},z_{j+1}\in Z$, which implies $z_{j+1}\in\low{z_{i-1}}$ or $z_{i-1}\in\low{z_{j+1}}$.
It follows that $z_0,\ldots,z_{i-1},z_{j+1},\ldots,z_k$ is a shorter sequence with the two properties, which contradicts the choice of $z_0,\ldots,z_k$.
Therefore, there is an index $j$ such that $\low{z_0}\geq\cdots\geq\low{z_j}\leq\cdots\leq\low{z_k}$.
It follows that $\low{z_j}\in\calB_{h-1}(\low{x})\cap\calB_{h-1}(\low{y})$.
Hence $\low{x}\geq A(\low{x},\low{y})\geq\low{z_j}$, so there is an index $i$ with $0\leq i\leq j$ such that $z_i\in A(\low{x},\low{y})$.
\end{proof}

Now, we are going to color incomparable pairs that are in a sense ``far away'' in $T$.
Let $\ell$ denote the depth-first search labeling of $T$ that visits the children of every node in the left-to-right order, and let $r$ denote the depth-first search labeling of $T$ that visits the children of every node in the right-to-left order, where a vertex receives its label when it is visited for the first time in the depth-first search.
Define the following four sets of incomparable pairs:
\begin{alignat*}{7}
I_1&=\{(x,y)\in\Inc(P)\colon&\ell(&\low{x})&&<{}&\ell&(Y)&&\text{ for every }&Y&\in\calB_{h-1}(\low{y})&&\},\\
I_2&=\{(x,y)\in\Inc(P)\colon&r(&\low{x})&&<{}&r&(Y)&&\text{ for every }&Y&\in\calB_{h-1}(\low{y})&&\},\\
I_3&=\{(x,y)\in\Inc(P)\colon&\ell(&\low{y})&&<{}&\ell&(X)&&\text{ for every }&X&\in\calB_{h-1}(\low{x})&&\},\\
I_4&=\{(x,y)\in\Inc(P)\colon&r(&\low{y})&&<{}&r&(X)&&\text{ for every }&X&\in\calB_{h-1}(\low{x})&&\}.
\end{alignat*}

\begin{lemma}
\label{lem:I_1-I_4}
None of the sets\/ $I_1$, $I_2$, $I_3$, $I_4$ contains an alternating cycle.
\end{lemma}

\begin{proof}
Suppose there is an alternating cycle $\{(x_i,y_i)\colon 1\leq i\leq k\}\subseteq I_1$.
For every $i$, by Lemma \ref{lem:comp}, $A(\low{x_i},\low{y_{i+1}})$ is defined, and we have $\ell(\low{x_{i+1}})<\ell(A(\low{x_i},\low{y_{i+1}}))\leq\ell(\low{x_i})$, as $(x_{i+1},y_{i+1})\in I_1$.
This is a contradiction, as the vertex subscripts are assumed to go cyclically over $\{1,\ldots,k\}$.
Similar arguments show that there are no alternating cycles in $I_2$, $I_3$ and $I_4$.
\end{proof}

Let $I=\Inc(P)\setminus(I_1\cup I_2\cup I_3\cup I_4)$.
To complete the proof of Lemma \ref{lem:main}, it remains to construct a valid coloring of the incomparable pairs in $I$.

\begin{lemma}
\label{lem:I}
For every\/ $(x,y)\in I$, we have\/ $\calB_h(\low{x})\cap\calB_h(\low{y})\neq\emptyset$, so\/ $A(\low{x},\low{y})$ is defined.
\end{lemma}

\begin{proof}
Let $(x,y)\in I$.
Since $(x,y)\notin I_1\cup I_2$, there are $Y_1,Y_2\in\calB_{h-1}(\low{y})$ such that $\ell(\low{x})\geq\ell(Y_1)$ and $r(\low{x})\geq r(Y_2)$.
We have $Y_1\leq Y_2$ or $Y_1\geq Y_2$.
Let $Y=\min\{Y_1,Y_2\}\in\calB_{h-1}(\low{y})$.
It follows that $\ell(\low{x})\geq\ell(Y)$ and $r(\low{x})\geq r(Y)$.
This implies $\low{x}\geq Y$.
Similarly, there is $X\in\calB_{h-1}(\low{x})$ such that $\low{y}\geq X$.
Since $\low{x},\low{y}\geq X,Y$, we have $X\leq Y$ or $X\geq Y$.
If $X\geq Y$, then the highest bag $X'\in\calB_{h-1}(\low{y})$ with $X\geq X'$ satisfies $X\close_1X'$ and hence $\low{x}\close_hX'$.
Similarly, if $Y\geq X$, then $\low{y}\close_hY'$.
In either case, we conclude that $\calB_h(\low{x})\cap\calB_h(\low{y})\neq\emptyset$.
\end{proof}

Lemma \ref{lem:B-size} implies that there is a coloring $\phi$ of the bags of $T$ with at most $1+s+\cdots+s^{2h}$ colors such that no two bags in the relation $\close_{2h}$ have the same color.
We can construct such a coloring greedily by processing the bags bottom-up and letting each bag $X$ be assigned a color that has not been used on the bags in $\calB_{2h}(X)\setminus\{X\}$.
From now on, the \emph{color} of a bag will always refer to the color assigned to it by $\phi$.
The following asserts that the color assigned by $\phi$ uniquely determines a member of $\calB_{2h}(X)$.

\begin{lemma}[cf.\ Lemma 5.4 in \cite{JMM+16}]
\label{lem:phi}
If\/ $X\close_{2h}Y$, $X\close_{2h}Z$, and\/ $\phi(Y)=\phi(Z)$, then\/ $Y=Z$.
\end{lemma}

\begin{proof}
Suppose $Y\neq Z$.
Since $X\close_{2h}Y$ and $X\close_{2h}Z$, we have $Y\close_{2h}Z$ or $Z\close_{2h}Y$.
Whichever of these holds, the definition of $\phi$ yields $\phi(Y)\neq\phi(Z)$.
\end{proof}

Let $\tau(X)$ denote the sequence of the colors of the members of $\calB_h(X)$ listed in the decreasing (top-to-bottom) order of $T$.
In particular, $\phi(X)$ occurs first in $\tau(X)$.
For two such sequences $\tau_1$ and $\tau_2$, let $\tau_1\oplus\tau_2$ denote the sequence obtained from $\tau_1$ and $\tau_2$ by the following procedure:
\begin{itemize}
\item remove from $\tau_1$ all colors that do not occur in $\tau_2$, thus obtaining $\tau_1'$;
\item remove from $\tau_2$ all colors that do not occur in $\tau_1$, thus obtaining $\tau_2'$;
\item let $\tau_1\oplus\tau_2$ be the longest common suffix (trailing segment) of $\tau_1'$ and $\tau_2'$.
\end{itemize}

\begin{lemma}
\label{lem:tau}
For any two bags\/ $X$ and\/ $Y$, the sequence of the colors of the members of\/ $\calB_h(X)\cap\calB_h(Y)$ listed in the decreasing (top-to-bottom) order of\/ $T$ is a suffix of\/ $\tau(X)\oplus\tau(Y)$.
\end{lemma}

\begin{proof}
Let $Z\in\calB_h(X)\cap\calB_h(Y)$.
The color $\phi(Z)$ occurs in both $\tau(X)$ and $\tau(Y)$.
Suppose $Z'$ is a bag in $\calB_h(X)$ or $\calB_h(Y)$ such that $Z>Z'$.
It follows that $Z\close_hZ'$, so $Z'\in\calB_{2h}(X)\cap\calB_{2h}(Y)$.
By Lemma \ref{lem:phi}, no member of $\calB_{2h}(X)\cup\calB_{2h}(Y)$ other than $Z'$ can have color $\phi(Z')$.
Therefore,
\begin{itemize}
\item if $Z'\in\calB_h(X)\cap\calB_h(Y)$, then $\phi(Z')$ occurs after $\phi(Z)$ in both $\tau(X)$ and $\tau(Y)$;
\item if $Z'\in\calB_h(X)\setminus\calB_h(Y)$, then $\phi(Z')$ does not occur in $\tau(Y)$;
\item if $Z'\in\calB_h(Y)\setminus\calB_h(X)$, then $\phi(Z')$ does not occur in $\tau(X)$.
\end{itemize}
This shows that the colors of the bags in $\calB_h(X)\cap\calB_h(Y)$ listed in the decreasing order of $T$ indeed form a suffix of $\tau(X)\oplus\tau(Y)$.
\end{proof}

Fix an arbitrary strict total order $\prec$ on the bags of $T$.
For any bags $X$, $Y$, $Z$ of $T$ such that $Z\in\calB_h(X)\cap\calB_h(Y)$, let $p_Z(X,Y)$ denote the color preceding $\phi(Z)$ in $\tau(X)\oplus\tau(Y)$ or a special value $\bot$ if $\phi(Z)$ is the first color in $\tau(X)\oplus\tau(Y)$, and define
\begin{equation*}
t_Z(X,Y)=\begin{cases}
1&\begin{minipage}[t]{\widthof{if $p_Z(X,Y)\neq\bot$ and $X'\prec Y'$, where $X'$ and $Y'$ are the unique bags such}}
if $p_Z(X,Y)\neq\bot$ and $X'\prec Y'$, where $X'$ and $Y'$ are the unique bags such that $X\close_hX'>Z$, $Y\close_hY'>Z$, and $\phi(X')=\phi(Y')=p_Z(X,Y)$,\strut
\end{minipage}\\
2&\text{if $p_Z(X,Y)\neq\bot$ and $X'\succ Y'$, where $X'$ and $Y'$ are as above,}\\
3&\text{if $p_Z(X,Y)=\bot$ or $X'=Y'$, where $X'$ and $Y'$ are as above.}
\end{cases}\\
\end{equation*}

Let $Z$ be a bag of $T$.
For every incomparable pair $(x,y)$ of $P_Z$, let $\sigma_Z(x,y)$ denote the color of $(x,y)$ in a valid coloring of $\Inc(P_Z)$ with colors $1,\ldots,d$.
Similarly, for every incomparable pair $(x,y)$ of $P'_Z$, let $\sigma'_Z(x,y)$ denote the color of $(x,y)$ in a valid coloring of $\Inc(P_Z)$ with colors $1,\ldots,d$.
Recall that every gadget in $P_Z$ and $P'_Z$ has at most $2^s$ vertices of the form $x_{Z,K,S}$.
Let every such vertex be assigned an identifier $\id_Z(x_{Z,K,S})\in\{1,\ldots,2^s\}$ so that for every adhesion set $K$ of $Z$, the mapping $x_{Z,K,S}\mapsto\id_Z(x_{Z,K,S})$ (with variable $S$) is injective.
For any $x,y\in P$, define
\begin{equation*}
\pi_Z(x,y)=\begin{cases}
1&\text{if $x,y\notin Z$ and $K_Z(x)=K_Z(y)$,}\\
2&\text{otherwise.}
\end{cases}
\end{equation*}

Now, for every $(x,y)\in I$, let the \emph{signature} $\Sigma(x,y)$ of $(x,y)$ comprise the following information:
\begin{enumerates}
\newcommand\mappingpair[1]{\makebox[\widthof{$\phi(Z)\mapsto\sigma_Z(\mu_Z(x),\nu_Z(y))$}][l]{$\phi(Z)\mapsto #1$}}
\item\label{sig:tau}    the pair of sequences $(\tau(\low{x}),\tau(\low{y}))$,
\item\label{sig:sigma}  the mapping (set of pairs) \mappingpair{\sigma_Z(\mu_Z(x),\nu_Z(y))}  for $Z\in\calB_h(\low{x})\cap\calB_h(\low{y})$,
\item\label{sig:sigma'} the mapping (set of pairs) \mappingpair{\sigma'_Z(\mu_Z(x),\nu_Z(y))} for $Z\in\calB_h(\low{x})\cap\calB_h(\low{y})$,
\item\label{sig:pi}     the mapping (set of pairs) \mappingpair{\pi_Z(x,y)}                   for $Z\in\calB_h(\low{x})\cap\calB_h(\low{y})$,
\item\label{sig:t}      the mapping (set of pairs) \mappingpair{t_Z(\low{x},\low{y})}         for $Z\in\calB_h(\low{x})\cap\calB_h(\low{y})$,
\item\label{sig:id}     the mapping (set of pairs) \mappingpair{\id_Z(\mu_Z(x))}              for $Z\in\calB_h(\low{x})\cap\calB_h(\low{y})\setminus\{\low{x}\}$.
\end{enumerates}
It follows from Lemma \ref{lem:phi} that every color in the codomain of $\phi$ is assigned at most one value by each of the mappings \ref{sig:sigma}--\ref{sig:id}.
It is clear that the number of distinct signatures is bounded in terms of $h$, $s$ and $d$.
We use $\Sigma(x,y)$ as the color of $(x,y)$ in the requested coloring of $I$ with a bounded number of colors.
To complete the proof of Lemma \ref{lem:main}, it remains to prove the following.

\begin{lemma}
\label{lem:signature}
For any\/ $\Sigma$, the set\/ $\{(x,y)\in I\colon\Sigma(x,y)=\Sigma\}$ contains no alternating cycle.
\end{lemma}

\begin{proof}
Suppose to the contrary that there is an alternating cycle $\{(x_i,y_i)\colon 1\leq i\leq k\}\subseteq I$ such that $\Sigma(x_i,y_i)=\Sigma$ for every $i$.
Assume without loss of generality that it is a shortest alternating cycle with this property.
For the rest of the proof, all the subscripts in the notation of vertices of the alternating cycle are assumed to go cyclically over $\{1,\ldots,k\}$.
By Lemma \ref{lem:comp} and Lemma \ref{lem:I}, $A(\low{x_i},\low{y_i})$ and $A(\low{x_i},\low{y_{i+1}})$ are defined for every $i$.
Let $Z$ be the highest bag such that $Z\leq A(\low{x_i},\low{y_i})$ and $Z\leq A(\low{x_i},\low{y_{i+1}})$ for every $i$.

Suppose $Z<A(\low{x_i},\low{y_i})$ and $Z<A(\low{x_i},\low{y_{i+1}})$ for every $i$.
It follows that
\begin{itemize}
\item if $Z<Z'\leq\low{y_i}$ and $ZZ'\in E(T)$, then $Z'\leq A(\low{x_i},\low{y_i})\leq\low{x_i}$,
\item if $Z<Z'\leq\low{x_i}$ and $ZZ'\in E(T)$, then $Z'\leq A(\low{x_i},\low{y_{i+1}})\leq\low{y_{i+1}}$.
\end{itemize}
Straightforward induction along the alternating cycle yields $Z'\leq\low{x_i}$ and $Z'\leq\low{y_i}$ for a unique bag $Z'$ with $Z<Z'$ and $ZZ'\in E(T)$.
This contradicts the choice of $Z$.
Therefore, we have $Z=A(\low{x_i},\low{y_i})$ or $Z=A(\low{x_i},\low{y_{i+1}})$ for some $i$.
In particular, we have $\low{x_{i_1}}\close_hZ$ for some $i_1$ and $\low{y_{i_2}}\close_hZ$ for some $i_2$.
For every $i$, since $\Sigma(\low{x_i},\low{y_i})=\Sigma=\Sigma(\low{x_{i_1}},\low{y_{i_1}})=\Sigma(\low{x_{i_2}},\low{y_{i_2}})$, we have $\tau(\low{x_i})=\tau(\low{x_{i_1}})$ and $\tau(\low{y_i})=\tau(\low{y_{i_2}})$ (cf.\ \ref{sig:tau}), so $\phi(Z)$ occurs in both $\tau(\low{x_i})$ and $\tau(\low{y_i})$.

Suppose $\low{y_i}\close_hZ$.
Since $\low{y_i}\close_hA(\low{x_i},\low{y_i})$ and $\low{y_i}\geq A(\low{x_i},\low{y_i})\geq Z$, Lemma \ref{lem:close} \ref{item:close5} yields $A(\low{x_i},\low{y_i})\close_hZ$ and hence $\low{x_i}\close_{2h}Z$.
Since $\phi(Z)$ occurs in $\tau(\low{x_i})$, there is a bag $Z'$ with $\low{x_i}\close_hZ'$ and $\phi(Z')=\phi(Z)$.
Lemma \ref{lem:phi} yields $Z'=Z$, so $\low{x_i}\close_hZ$.
The same argument shows that $\low{x_i}\close_hZ$ implies $\low{y_{i+1}}\close_hZ$.
Straightforward induction yields $\low{x_i}\close_hZ$ and $\low{y_i}\close_hZ$ for every $i$.

In the rest of the proof, we will make essential use of the fact that the information assigned to $\phi(Z)$ by each of the mappings \ref{sig:sigma}--\ref{sig:id} saved in $\Sigma(x_i,y_i)$ is the same for every $i$.
In particular, $\pi_Z(x_i,y_i)$ is the same for every $i$ (cf.\ \ref{sig:pi}).

\emph{Case 1.} $\pi_Z(x_i,y_i)=1$ for every $i$.
It follows that $\low{x_i},\low{y_i}>Z$ and $K_Z(x_i)=K_Z(y_i)$ for every $i$.
Suppose that for some $i$, there is $z\in Z$ with $x_i\leq_Pz\leq_Py_{i+1}$, but there is no $z'\in Z$ with $x_{i-1}\leq_Pz'\leq_Py_i$.
Lemma \ref{lem:comp} yields $A(\low{x_{i-1}},\low{y_i})>Z$ and thus $K_Z(x_{i-1})=K_Z(y_i)=K_Z(x_i)$.
Hence $\mu_Z(x_{i-1})=x_{Z,K,S}$ and $\mu_Z(x_i)=x_{Z,K,S'}$, where $K=K_Z(x_{i-1})=K_Z(x_i)$, $S=Z\cap\upset_Px_{i-1}$, and $S'=Z\cap\upset_Px_i$.
This and $\id_Z(\mu_Z(x_{i-1}))=\id_Z(\mu_Z(x_i))$ (cf.\ \ref{sig:id}) yield $\id_Z(x_{Z,K,S})=\id_Z(x_{Z,K,S'})$ and hence $S=S'$.
Since $x_i\leq_Pz$, we have $z\in S'=S$, so $x_{i-1}\leq_Pz\leq_Py_{i+1}$.
This implies that $k>2$ and $\{(x_j,y_j)\colon 1\leq j\leq k$ and $j\neq i\}$ is a shorter alternating cycle in $I$, a contradiction.
Therefore, there is $z\in Z$ with $x_i\leq_Pz\leq_Py_{i+1}$ either for every $i$ or for no $i$.

\emph{Subcase 1.1.} For every $i$, there is $z\in Z$ with $x_i\leq_Pz\leq_Py_{i+1}$.
By Lemma \ref{lem:aux} \ref{item:aux1}, we have $(\mu_Z(x_i),\nu_Z(y_i))\in\Inc(P_Z)$ for every $i$.
By Lemma \ref{lem:aux} \ref{item:aux3}, we have $\mu_Z(x_i)\leq_{P_Z}\nu_Z(y_{i+1})$ for every $i$.
Hence $\{(\mu_Z(x_i),\nu_Z(y_i))\colon 1\leq i\leq k\}$ is an alternating cycle in $\Inc(P_Z)$ (see the left side of Figure~\ref{fig:gadget}).
However, it is monochromatic in $\sigma_Z$ (cf.\ \ref{sig:sigma}).
This contradicts the assumption that $\sigma_Z$ is a valid coloring of $\Inc(P_Z)$.

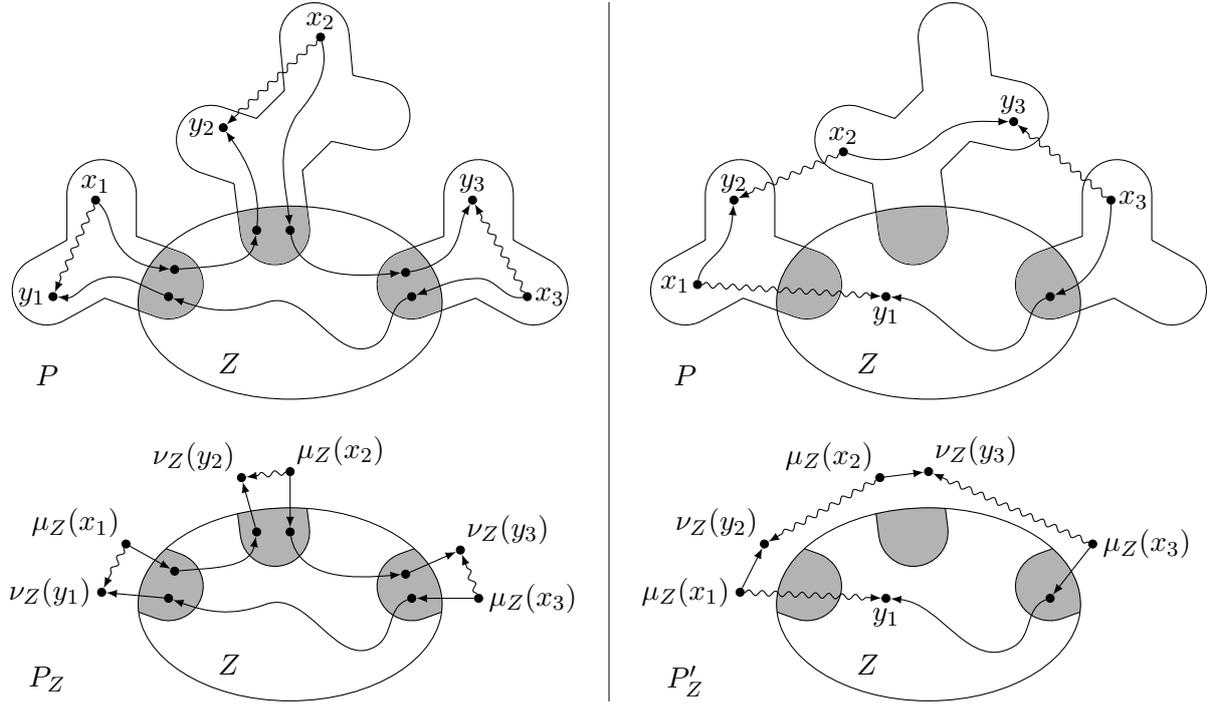
\begin{figure}[t]
\centering
\begin{tikzpicture}[scale=.8,>=latex]
\begin{scope}
  \coordinate (a1) at (-2,0.3);
  \coordinate (a2) at (-3.1,0.7);
  \coordinate (a3) at (-4,0.2);
  \coordinate (a4) at (-3.1,1.8);
  \coordinate (b1) at (-0.25,1.2);
  \coordinate (b2) at (-0.4,2.4);
  \coordinate (b3) at (-1.3,2.7);
  \coordinate (b4) at (0.5,3.3);
  \coordinate (b5) at (0.4,4.4);
  \coordinate (b6) at (1.4,3.1);
  \coordinate (c1) at (2,0.3);
  \coordinate (c2) at (3.1,0.7);
  \coordinate (c3) at (4,0.2);
  \coordinate (c4) at (3.1,1.8);
  \draw[line width=26.4pt,line cap=round] (a1)--(a2)--(a3) (a2)--(a4);
  \draw[line width=26.4pt,line cap=round] (b1)--(b2)--(b3) (b2)--(b4)--(b5) (b4)--(b6);
  \draw[line width=26.4pt,line cap=round] (c1)--(c2)--(c3) (c2)--(c4);
  \draw[white,line width=25.6pt,line cap=round] (a1)--(a2)--(a3) (a2)--(a4);
  \draw[white,line width=25.6pt,line cap=round] (b1)--(b2)--(b3) (b2)--(b4)--(b5) (b4)--(b6);
  \draw[white,line width=25.6pt,line cap=round] (c1)--(c2)--(c3) (c2)--(c4);
  \begin{scope}
  \clip (0,0) ellipse (2.5cm and 1.6cm);
  \draw[black!30,line width=25.6pt,line cap=round] (a1)--(a2)--(a3) (a2)--(a4);
  \draw[black!30,line width=25.6pt,line cap=round] (b1)--(b2)--(b3) (b2)--(b4)--(b5) (b4)--(b6);
  \draw[black!30,line width=25.6pt,line cap=round] (c1)--(c2)--(c3) (c2)--(c4);
  \end{scope}
  \draw (0,0) ellipse (2.5cm and 1.6cm);
  \tikzstyle{every node}=[draw,fill,circle,minimum size=3pt,inner sep=0pt]
  \tikzstyle{every label}=[draw=none,fill=none,rectangle,inner sep=1pt]
  \node[label=above:$x_1$] (x1) at (-3.2,1.7) {};
  \node[label=left:$y_1$] (y1) at (-3.9,0.1) {};
  \node[label=above:$x_2$] (x2) at (0.5,4.4) {};
  \node[label=left:$y_2$] (y2) at (-1.1,2.9) {};
  \node[label=right:$x_3$] (x3) at (3.9,0.1) {};
  \node[label=above:$y_3$] (y3) at (3,1.7) {};
  \node (x1') at (-1.9,0.55) {};
  \node (y1') at (-2,0.1) {};
  \node (x2') at (0,1.2) {};
  \node (y2') at (-0.55,1.2) {};
  \node (x3') at (2,0.1) {};
  \node (y3') at (1.9,0.5) {};
  \path (x1) edge[->,decorate,decoration={snake,segment length=5pt,amplitude=1pt,post length=4pt}] (y1);
  \path (x2) edge[->,decorate,decoration={snake,segment length=5pt,amplitude=1pt,post length=4pt}] (y2);
  \path (x3) edge[->,decorate,decoration={snake,segment length=5pt,amplitude=1pt,post length=4pt}] (y3);
  \draw[->] (x1) to[out=-45,in=120] (-2.85,1) to[out=-60,in=175] (x1');
  \draw[->] (x1') to[out=5,in=-100] (y2');
  \draw[->] (y2') to[out=85,in=-55] (y2);
  \draw[->] (x2) to[out=-75,in=50] (0.1,3.1) to[out=-120,in=95] (x2');
  \draw[->] (x2') to[out=-80,in=-175] (y3');
  \draw[->] (y3') to[out=10,in=-140] (2.6,0.8) to[out=50,in=-120] (y3);
  \draw[->] (x3) to[out=-170,in=0] (3.1,0.4) to[out=180,in=20] (x3');
  \draw[->] (x3') to[out=-160,in=30] (1.4,-0.7) to[out=-150,in=20] (-0.4,0) to[out=-160,in=-20] (y1');
  \draw[->] (y1') to[out=160,in=15] (-3.1,0.4) to[out=-165,in=-5] (y1);
  \node[draw=none,fill=none] at (-1,-1) {$Z$};
  \node[draw=none,fill=none,text depth=0pt] at (-4,-1.2) {$P$};
\end{scope}
\begin{scope}[yshift=-5cm]
  \coordinate (a1) at (-2,0.3);
  \coordinate (a2) at (-3.1,0.7);
  \coordinate (b1) at (-0.25,1.2);
  \coordinate (b2) at (-0.4,2.4);
  \coordinate (c1) at (2,0.3);
  \coordinate (c2) at (3.1,0.7);
  \begin{scope}
  \clip (0,0) ellipse (2.5cm and 1.6cm);
  \draw[line width=26.4pt,line cap=round] (a1)--(a2);
  \draw[line width=26.4pt,line cap=round] (b1)--(b2);
  \draw[line width=26.4pt,line cap=round] (c1)--(c2);
  \draw[black!30,line width=25.6pt,line cap=round] (a1)--(a2);
  \draw[black!30,line width=25.6pt,line cap=round] (b1)--(b2);
  \draw[black!30,line width=25.6pt,line cap=round] (c1)--(c2);
  \end{scope}
  \draw (0,0) ellipse (2.5cm and 1.6cm);
  \tikzstyle{every node}=[draw,fill,circle,minimum size=3pt,inner sep=0pt]
  \tikzstyle{every label}=[draw=none,fill=none,rectangle,inner sep=1pt]
  \node[label=above left:$\mu_Z(x_1)$] (x1) at (-2.7,1) {};
  \node[label={[label distance=1pt]left:$\!\nu_Z(y_1)$}] (y1) at (-3.1,0.2) {};
  \node[label=above right:$\mu_Z(x_2)$] (x2) at (0,2.2) {};
  \node[label=above left:$\nu_Z(y_2)$] (y2) at (-0.8,2.1) {};
  \node[label={[label distance=1pt]right:$\mu_Z(x_3)$}] (x3) at (3.1,0.1) {};
  \node[label=above right:$\nu_Z(y_3)$] (y3) at (2.8,0.9) {};
  \node (x1') at (-1.9,0.55) {};
  \node (y1') at (-2,0.1) {};
  \node (x2') at (0,1.2) {};
  \node (y2') at (-0.55,1.2) {};
  \node (x3') at (2,0.1) {};
  \node (y3') at (1.9,0.5) {};
  \path (x1) edge[->,decorate,decoration={snake,segment length=5pt,amplitude=1pt,post length=4pt}] (y1);
  \path (x2) edge[->,decorate,decoration={snake,segment length=5pt,amplitude=1pt,post length=4pt}] (y2);
  \path (x3) edge[->,decorate,decoration={snake,segment length=5pt,amplitude=1pt,post length=4pt}] (y3);
  \draw[->] (x1) to (x1');
  \draw[->] (x1') to[out=5,in=-100] (y2');
  \draw[->] (y2') to (y2);
  \draw[->] (x2) to (x2');
  \draw[->] (x2') to[out=-80,in=-175] (y3');
  \draw[->] (y3') to (y3);
  \draw[->] (x3) to (x3');
  \draw[->] (x3') to[out=-160,in=30] (1.4,-0.7) to[out=-150,in=20] (-0.4,0) to[out=-160,in=-20] (y1');
  \draw[->] (y1') to (y1);
  \node[draw=none,fill=none] at (-1,-1) {$Z$};
  \node[draw=none,fill=none,text depth=0pt] at (-4,-1.2) {$P_Z$};
\end{scope}
\begin{scope}[xshift=10.5cm]
  \coordinate (a1) at (-2,0.3);
  \coordinate (a2) at (-3.1,0.7);
  \coordinate (a3) at (-4,0.2);
  \coordinate (a4) at (-3.1,1.8);
  \coordinate (b1) at (-0.25,1.2);
  \coordinate (b2) at (-0.4,2.4);
  \coordinate (b3) at (-1.3,2.7);
  \coordinate (b4) at (0.5,3.3);
  \coordinate (b5) at (0.4,4.4);
  \coordinate (b6) at (1.4,3.1);
  \coordinate (c1) at (2,0.3);
  \coordinate (c2) at (3.1,0.7);
  \coordinate (c3) at (4,0.2);
  \coordinate (c4) at (3.1,1.8);
  \draw[line width=26.4pt,line cap=round] (a1)--(a2)--(a3) (a2)--(a4);
  \draw[line width=26.4pt,line cap=round] (b1)--(b2)--(b3) (b2)--(b4)--(b5) (b4)--(b6);
  \draw[line width=26.4pt,line cap=round] (c1)--(c2)--(c3) (c2)--(c4);
  \draw[white,line width=25.6pt,line cap=round] (a1)--(a2)--(a3) (a2)--(a4);
  \draw[white,line width=25.6pt,line cap=round] (b1)--(b2)--(b3) (b2)--(b4)--(b5) (b4)--(b6);
  \draw[white,line width=25.6pt,line cap=round] (c1)--(c2)--(c3) (c2)--(c4);
  \begin{scope}
  \clip (0,0) ellipse (2.5cm and 1.6cm);
  \draw[black!30,line width=25.6pt,line cap=round] (a1)--(a2)--(a3) (a2)--(a4);
  \draw[black!30,line width=25.6pt,line cap=round] (b1)--(b2)--(b3) (b2)--(b4)--(b5) (b4)--(b6);
  \draw[black!30,line width=25.6pt,line cap=round] (c1)--(c2)--(c3) (c2)--(c4);
  \end{scope}
  \draw (0,0) ellipse (2.5cm and 1.6cm);
  \tikzstyle{every node}=[draw,fill,circle,minimum size=3pt,inner sep=0pt]
  \tikzstyle{every label}=[draw=none,fill=none,rectangle,inner sep=1pt]
  \node[label=left:$x_1$] (x1) at (-3.8,0.3) {};
  \node[label={[label distance=1pt]below:$y_1$}] (y1) at (-0.7,0.1) {};
  \node[label=above:$x_2$] (x2) at (-1.4,2.5) {};
  \node[label=above:$y_2$] (y2) at (-3.2,1.7) {};
  \node[label=right:$x_3$] (x3) at (3,1.7) {};
  \node[label=above:$y_3$] (y3) at (1.4,3) {};
  \node (x3') at (2,0.1) {};
  \path (x1) edge[->,decorate,decoration={snake,segment length=5pt,amplitude=1pt,post length=4pt}] (y1);
  \path (x2) edge[->,decorate,decoration={snake,segment length=5pt,amplitude=1pt,post length=4pt}] (y2);
  \path (x3) edge[->,decorate,decoration={snake,segment length=5pt,amplitude=1pt,post length=4pt}] (y3);
  \draw[->] (x1) to[out=80,in=-135] (-3.5,0.8) to[out=45,in=-95] (y2);
  \draw[->] (x2) to[out=-20,in=-135] (0.2,2.8) to[out=35,in=180] (y3);
  \draw[->] (x3) to[out=-90,in=20] (x3');
  \draw[->] (x3') to[out=-160,in=30] (1.4,-0.7) to[out=-150,in=-5] (y1);
  \node[draw=none,fill=none] at (-1,-1) {$Z$};
  \node[draw=none,fill=none,text depth=0pt] at (-4,-1.2) {$P$};
\end{scope}
\begin{scope}[xshift=10.5cm,yshift=-5cm]
  \coordinate (a1) at (-2,0.3);
  \coordinate (a2) at (-3.1,0.7);
  \coordinate (b1) at (-0.25,1.2);
  \coordinate (b2) at (-0.4,2.4);
  \coordinate (c1) at (2,0.3);
  \coordinate (c2) at (3.1,0.7);
  \begin{scope}
  \clip (0,0) ellipse (2.5cm and 1.6cm);
  \draw[line width=26.4pt,line cap=round] (a1)--(a2);
  \draw[line width=26.4pt,line cap=round] (b1)--(b2);
  \draw[line width=26.4pt,line cap=round] (c1)--(c2);
  \draw[black!30,line width=25.6pt,line cap=round] (a1)--(a2);
  \draw[black!30,line width=25.6pt,line cap=round] (b1)--(b2);
  \draw[black!30,line width=25.6pt,line cap=round] (c1)--(c2);
  \end{scope}
  \draw (0,0) ellipse (2.5cm and 1.6cm);
  \tikzstyle{every node}=[draw,fill,circle,minimum size=3pt,inner sep=0pt]
  \tikzstyle{every label}=[draw=none,fill=none,rectangle,inner sep=1pt]
  \node[label={[label distance=1pt]left:$\mu_Z(x_1)$}] (x1) at (-3.1,0.2) {};
  \node[label={[label distance=1pt]below:$y_1$}] (y1) at (-0.7,0.1) {};
  \node[label=above left:$\mu_Z(x_2)$] (x2) at (-0.8,2.1) {};
  \node[label=above left:$\nu_Z(y_2)$] (y2) at (-2.7,1) {};
  \node[label={[label distance=1pt]right:$\mu_Z(x_3)$}] (x3) at (2.7,1) {};
  \node[label=above right:$\nu_Z(y_3)$] (y3) at (0,2.2) {};
  \node (x3') at (2,0.1) {};
  \path (x1) edge[->,decorate,decoration={snake,segment length=5pt,amplitude=1pt,post length=4pt}] (y1);
  \path (x2) edge[->,decorate,decoration={snake,segment length=5pt,amplitude=1pt,post length=4pt}] (y2);
  \path (x3) edge[->,decorate,decoration={snake,segment length=5pt,amplitude=1pt,post length=4pt}] (y3);
  \draw[->] (x1) to (y2);
  \draw[->] (x2) to (y3);
  \draw[->] (x3) to (x3');
  \draw[->] (x3') to[out=-160,in=30] (1.4,-0.7) to[out=-150,in=-5] (y1);
  \node[draw=none,fill=none] at (-1,-1) {$Z$};
  \node[draw=none,fill=none,text depth=0pt] at (-4,-1.2) {$P'_Z$};
\end{scope}
\draw[very thin] (current bounding box.south)--(current bounding box.north);
\end{tikzpicture}
\caption{Illustration of the proof of Lemma \ref{lem:signature}: an alternating cycle in $I$ and the corresponding alternating cycle in $\Inc(P_Z)$ (subcase 1.1 in the proof, left side of the figure) or $\Inc(P'_Z)$ (case 2 in the proof, right side of the figure)}
\label{fig:gadget}
\end{figure}

\emph{Subcase 1.2.} For every $i$, there is no $z\in Z$ with $x_i\leq_Pz\leq_Py_{i+1}$.
By Lemma \ref{lem:comp}, we have $A(\low{x_i},\low{y_{i+1}})>Z$ for every $i$.
Let $Z'_i$ be the lowest bag in $\calB_h(\low{x_i})\cap\calB_h(\low{y_{i+1}})$ with $Z'_i>Z$.
We have $\tau(\low{x_1})=\cdots=\tau(\low{x_k})$ and $\tau(\low{y_1})=\cdots=\tau(\low{y_k})$ (cf.\ \ref{sig:tau}), which determines a color $c$ such that $c=p_Z(\low{x_i},\low{y_i})=p_Z(\low{x_i},\low{y_{i+1}})$ for all $i$ (in particular, $c\neq\bot$).
By Lemma \ref{lem:tau}, we have $\phi(Z'_i)=c$ for every $i$.
Since $A(\low{x_i},\low{y_{i+1}})>Z$ for every $i$, we have $A(\low{x_{i_0}},\low{y_{i_0}})=Z$ for some $i_0$.
This and $p_Z(\low{x_{i_0}},\low{y_{i_0}})\neq\bot$ imply $t_Z(\low{x_{i_0}},\low{y_{i_0}})\in\{1,2\}$.
For every $i$, we have $t_Z(\low{x_{i+1}},\low{y_{i+1}})=t_Z(\low{x_{i_0}},\low{y_{i_0}})$ (cf.\ \ref{sig:t}), which implies $Z'_{i+1}\prec Z'_i$ when $t_Z(\low{x_{i_0}},\low{y_{i_0}})=1$ and $Z'_{i+1}\succ Z'_i$ when $t_Z(\low{x_{i_0}},\low{y_{i_0}})=2$.
This is a contradiction in either case, as the subscripts are assumed to go cyclically over $\{1,\ldots,k\}$ and $\prec$ is a strict order.

\emph{Case 2.} $\pi_Z(x_i,y_i)=2$ for every $i$.
We have $(\mu_Z(x_i),\nu_Z(y_i))\in\Inc(P'_Z)$ and $\mu_Z(x_i)\leq_{\smash[b]{P'_Z}}\nu_Z(y_{i+1})$ for every $i$, by Lemma \ref{lem:aux} \ref{item:aux2} and \ref{item:aux4}.
Hence $\{(\mu_Z(x_i),\nu_Z(y_i))\colon 1\leq i\leq k\}$ is an alternating cycle in $\Inc(P'_Z)$ (see the right side of Figure~\ref{fig:gadget}).
However, it is monochromatic in $\sigma'_Z$ (cf.\ \ref{sig:sigma'}).
This contradicts the assumption that $\sigma'_Z$ is a valid coloring of $\Inc(P'_Z)$.
\end{proof}

\subsection{Summary}

Let $\calG$ be a proper topologically closed class of graphs, and let $h\in\setN$.
Let $P$ be a poset of height at most $h$ with cover graph $G\in\calG$.
By Theorem \ref{thm:grohe-marx}, we have $\calG\subseteq\calT_s(\calA_t(\calL_f\cup\calD_d))$, where $s$, $t$, $f$ and $d$ depend only on $\calG$.
By Corollary \ref{cor:gadget}, the gadget extensions $P_Z$ and $P'_Z$ for all bags $Z$ of the tree decomposition of $G$ witnessing $G\in\calT_s(\calA_t(\calL_f\cup\calD_d))$ have dimension bounded in terms of $h$ and $\calG$.
Therefore, by Lemma \ref{lem:main}, the poset $P$ has dimension bounded in terms of $h$ and $\calG$.
This completes the proof of Theorem \ref{thm:main}.

Theorem \ref{thm:tree-width} plays an important role in the proof of Theorem \ref{thm:main} above.
However, we can use Lemma \ref{lem:main} to reprove Theorem \ref{thm:tree-width} as well.
Fix $k\in\setN$.
If $\calV_{k+1}$ denotes the class of graphs with at most $k+1$ vertices, then $\calT_k(\calV_{k+1})$ is the class of graphs of tree-width at most $k$.
The graphs in $\calE_k(\calV_{k+1})$ have bounded size, so posets with cover graphs in $\calE_k(\calV_{k+1})$ have bounded dimension.
By Lemma \ref{lem:main}, for any $h,k\in\setN$, posets of height at most $h$ with cover graphs in $\calT_k(\calV_{k+1})$ have bounded dimension, which is exactly the statement of Theorem \ref{thm:tree-width}.

\section{Proof of Lemma \ref{lem:diameter}}
\label{sec:diameter}

\begin{proof}[Proof of Lemma \ref{lem:diameter}]
Let $P$ be a poset of height at most $h$ with cover graph $G\in\calG$.
If $P$ is the disjoint union of posets $P_1,\ldots,P_k$, where $k\geq 2$, then $\dim(P)=\max\{\dim(P_1),\ldots,\dim(P_k),2\}$ \cite{Hir55}.
Therefore, since our goal is to bound $\dim(P)$, we can assume without loss of generality that $G$ is connected.
Let $v$ be an arbitrary minimal vertex in $P$.
Let $A_0,\ldots,A_n$ be the distance sets of the comparability graph of $P$ in the breadth-first search starting at $v$.
That is, $A_0=\{v\}$, $A_i=\{x\in P\setminus(A_0\cup\cdots\cup A_{i-1})\colon$there is $y\in A_{i-1}$ comparable to $x$ in $P\}$ for $i\geq 1$, and $n$ is greatest such that $A_n\neq\emptyset$.
The following properties are straightforward:
\begin{enumerate}
\item\label{item:diameter1} if $x\in A_i$, $y\in A_{i-1}$, $x$ and $y$ are comparable, and $i$ is odd, then $x\geq_Py$,
\item if $x\in A_i$, $y\in A_{i-1}$, $x$ and $y$ are comparable, and $i$ is even, then $x\leq_Py$,
\item\label{item:diameter3} if $x\in A_i$, $y\in A_j$, and $|i-j|\geq 2$, then $x$ and $y$ are incomparable,
\item if $x\in A_i$ and $i$ is odd, then there is $y\in A_{i-1}$ such that $x\geq_Py$,
\item if $x\in A_i$ and $i$ is even, then $x=v$ or there is $y\in A_{i-1}$ such that $x\leq_Py$.
\end{enumerate}

Let $P_1=P[A_0\cup A_1]$.
It follows that if $G_1$ denotes the cover graph of $P_1$, then $G_1=G[A_0\cup A_1]$ and thus $G_1\in\calG$.
For $2\leq i\leq n$, let $P_i$ be the poset obtained from $P[A_0\cup\cdots\cup A_i]$ by contracting $A_0\cup\cdots\cup A_{i-2}$ to a single vertex.
That is, the whole set $A_0\cup\cdots\cup A_{i-2}$ is replaced by a single vertex $v'$ such that
\begin{itemize}
\item $v'\leq_{P_i}x$ for every $x\in A_{i-1}$ if $i$ is even,
\item $v'\geq_{P_i}x$ for every $x\in A_{i-1}$ if $i$ is odd,
\item $v'$ is incomparable with every vertex in $A_i$.
\end{itemize}
It follows that the cover graph $G_i$ of $P_i$ is obtained from $G[A_0\cup\cdots\cup A_i]$ by contracting $A_0\cup\cdots\cup A_{i-2}$ to the single vertex $v'$.
Since $G[A_0\cup\cdots\cup A_{i-2}]$ is connected, $G_i$ is a minor of $G$, so $G_i\in\calG$.
For $1\leq i\leq n$, the height of $P_i$ is at most $h$, and hence every comparability in $P_i$ is witnessed by a path in $G_i$ of length at most $h-1$.
Therefore, the distance from $v'$ (or from $v$ if $i=1$) of every vertex in $G_i$ is at most $2h-2$, so $G_i\in\calG_{2h-2}$.
This and the assumption of the lemma imply that $P_i$ has bounded dimension.

Let $d$ be a common bound on the dimensions of all the $P_i$.
We are ready to construct a valid coloring of $\Inc(P)$ with a bounded number of colors.
First, observe that by the property \ref{item:diameter3} above, neither of the following two sets of incomparable pairs contains an alternating cycle:
\begin{align*}
I_1&=\{(x,y)\in\Inc(P)\colon\text{$x\in A_i$ and $y\in A_j$ for some $i$ and $j$ with $i-j\geq 2$}\},\\
I_2&=\{(x,y)\in\Inc(P)\colon\text{$x\in A_i$ and $y\in A_j$ for some $i$ and $j$ with $j-i\geq 2$}\}.
\end{align*}
Therefore, we can color $I_1\cup I_2$ with just two colors.
Every remaining incomparable pair $(x,y)\in\Inc(P)$ satisfies $x,y\in A_{i-1}\cup A_i$ for some $i$ and therefore it is also in $\Inc(P_i)$ for some $i$.
Fix a valid coloring of every $\Inc(P_i)$ with colors $1,\ldots,d$ if $i$ is even or $d+1,\ldots,2d$ if $i$ is odd.
We color the incomparable pairs in $\Inc(P)$ so that if $x,y\in A_{i-1}\cup A_i$ and $x\in A_i$ or $y\in A_i$, then $(x,y)$ is assigned the color that it has in $\Inc(P_i)$.
It is clear that this rule assigns a unique color to each incomparable pair in $\Inc(P)$.
Suppose that this coloring of $\Inc(P)$ yields a monochromatic alternating cycle $C=\{(x_j,y_j)\colon 1\leq j\leq k\}$.
Every incomparable pair $(x_j,y_j)\in C$ belongs to $A_{i-1}\cup A_i$ for some $i$, where the parity of $i$ is determined by the color of $C$ and thus is common for all $(x_j,y_j)\in C$.
This and the properties \ref{item:diameter1}--\ref{item:diameter3} imply that the index $i$ itself is common for all $(x_j,y_j)\in C$, which contradicts the assumption that we have chosen a valid coloring of $\Inc(P_i)$.
This shows that the coloring of $\Inc(P)$ that we have constructed is valid.
\end{proof}

\linespread{1.054}
\section*{Acknowledgment}

I am grateful to Gwenaël Joret and Tom Trotter for helpful discussions and comments.

\end{document}